\documentclass[a4paper, fleqn, abstract=true, headings=small, english, bibliography=totoc]{scrartcl}

\usepackage{setspace}
\usepackage{indentfirst} 

\usepackage[T1]{fontenc}
\usepackage[utf8]{inputenc}

\usepackage{amsfonts,amsmath,amssymb,esint}
\usepackage{mathtools}
\usepackage{amsthm}
\usepackage{thmtools}

\usepackage{mathrsfs}
\usepackage{stmaryrd}

\usepackage[dvipsnames]{xcolor}
\usepackage{tikz}

\usepackage{array}

\usepackage[alwaysadjust]{paralist}
\usepackage[inline]{enumitem}

\usepackage{babel}
\usepackage[theoremfont, osf]{newpxtext}
\usepackage[slantedGreek]{newpxmath}

\theoremstyle{plain}
\newtheorem{proposition}{\protect\propositionname}[section]
\newtheorem{lemma}[proposition]{\protect\lemmaname}

\newtheorem{theorem}[proposition]{\protect\theoremname}
\newtheorem{corollary}[proposition]{\protect\corollaryname}
\newtheorem{definition}[proposition]{\protect\definitionname}

\newtheorem{openproblem}[proposition]{\protect\openproblemname}

\usepackage{etoolbox}
\theoremstyle{definition}
\newtheorem{example}[proposition]{\protect\examplename}
\AtBeginEnvironment{example}{%
	\pushQED{\qed}%
}
\AtEndEnvironment{example}{\popQED\endexample}

\theoremstyle{remark}
\newtheorem{remark}[proposition]{\protect\remarkname}
\AtBeginEnvironment{remark}{%
	\pushQED{\qed}%
}
\AtEndEnvironment{remark}{\popQED\endremark}

\declaretheoremstyle[%
spaceabove=.3\topsep,%
spacebelow=.3\topsep,%
headfont=\small\itshape,%
headpunct={. ---},%
notefont=\normalfont\itshape,%
bodyfont=\itshape,%
headindent=\parindent,%
numbered=yes,%
]{mystyle1}%

\newenvironment{solution}{\noindent\emph{\protect\solutionname. }\pushQED{\qed}}{\popQED}

\newcommand{\propositionname}{}
\newcommand{\lemmaname}{}
\newcommand{\sublemmaname}{}
\newcommand{\theoremname}{}
\newcommand{\corollaryname}{}
\newcommand{\definitionname}{}
\newcommand{\notationname}{}
\newcommand{\claimname}{}
\newcommand{\examplename}{}
\newcommand{\remarkname}{}
\newcommand{\exercisename}{}
\newcommand{\openproblemname}{}
\newcommand{\propertyname}{}
\newcommand{\stepname}{}
\newcommand{\casename}{}
\newcommand{\Claimname}{}
\newcommand{\factname}{}
\newcommand{\Stepname}{}
\newcommand{\Casename}{}
\newcommand{\solutionname}{}

\addto\captionsenglish{%
	\renewcommand{\propositionname}{Proposition}%
	\renewcommand{\lemmaname}{Lemma}%
	\renewcommand{\sublemmaname}{Sublemma}%
	\renewcommand{\theoremname}{Theorem}%
	\renewcommand{\corollaryname}{Corollary}%
	\renewcommand{\definitionname}{Definition}%
	\renewcommand{\notationname}{Notation}%
	\renewcommand{\claimname}{Claim}%
	\renewcommand{\examplename}{Example}%
	\renewcommand{\remarkname}{Remark}%
	\renewcommand{\exercisename}{Exercise}%
	\renewcommand{\openproblemname}{Open problem}%
	\renewcommand{\propertyname}{Property}%
	\renewcommand{\stepname}{Step}%
	\renewcommand{\casename}{Case}%
	\renewcommand{\Claimname}{Claim}%
	\renewcommand{\factname}{Fact}%
	\renewcommand{\Stepname}{Step}%
	\renewcommand{\Casename}{Case}%
	\renewcommand{\solutionname}{Solution}%
}

\addto\captionsfrench{%
	\renewcommand{\propositionname}{Proposition}%
	\renewcommand{\lemmaname}{Lemme}%
	\renewcommand{\sublemmaname}{Sous-lemme}%
	\renewcommand{\theoremname}{Théorème}%
	\renewcommand{\corollaryname}{Corollaire}%
	\renewcommand{\definitionname}{Définition}%
	\renewcommand{\notationname}{Notation}%
	\renewcommand{\claimname}{Affirmation}%
	\renewcommand{\examplename}{Exemple}%
	\renewcommand{\remarkname}{Remarque}%
	\renewcommand{\exercisename}{Exercice}%
	\renewcommand{\openproblemname}{Problème ouvert}%
	\renewcommand{\propertyname}{Propriété}%
	\renewcommand{\stepname}{Étape}%
	\renewcommand{\casename}{Cas}%
	\renewcommand{\Claimname}{Affirmation}%
	\renewcommand{\factname}{Fait}%
	\renewcommand{\Stepname}{Étape}%
	\renewcommand{\Casename}{Cas}%
	\renewcommand{\solutionname}{Solution}%
}

\usepackage[shortcuts]{extdash}

\definecolor{purple}{cmyk}{0.75,0.90,0,0}
\definecolor{DB}{rgb}{0.07,0.0,0.5}
\definecolor{DG}{rgb}{0.0,0.37,0.07}%
\definecolor{DR}{rgb}{0.37,0,0.07}

\usepackage[
pdftitle={},%
pdfauthor={}, hyperindex=true, colorlinks=true, urlcolor={DR}, linkcolor={DR}, menucolor={DR}, citecolor={DR}, anchorcolor={DR},linktoc=all,pagebackref,final]{hyperref}

\newcommand{\intd}[1]{\mathop{}\!\diffd{#1}}

\newcommand{\diffd}{{\operatorfont d}}
\newcommand{\diffD}{{\operatorfont D}}

\newcommand{\bB}{\varmathbb{B}}

\newcommand{\bN}{\varmathbb{N}}

\newcommand{\bR}{\varmathbb{R}}
\newcommand{\bS}{\varmathbb{S}}

\newcommand{\mG}{\mathscr{G}}
\newcommand{\mH}{\mathscr{H}}

\newcommand{\mM}{\mathscr{M}}
\newcommand{\mN}{\mathscr{N}}

\newcommand{\mR}{\mathscr{R}}

\newcommand{\cE}{\mathcal{E}}

\newcommand{\st}{\mathpunct{:}}

\newcommand{\Dext}{{}\diffD}
\newcommand{\cs}{_\textnormal{c}}

\newcommand{\maximal}[2][]{\mathcal{M}\farg[#1]{#2}}

\newcommand{\cball}[4][]{\overline{\ifblank{#2}{\bB}{B_{#2}}\ifblank{#4}{}{^{#4}}}\ifblank{#3}{}{\parens[#1]{#3}}}
\newcommand{\ball}[4][]{\ifblank{#2}{\bB}{B_{#2}}\ifblank{#4}{}{^{#4}}\ifblank{#3}{}{\parens[#1]{#3}}}
\newcommand{\ccube}[4][]{\overline{Q\ifblank{#4}{}{^{#4}}\ifblank{#2}{}{_{#2}}\ifblank{#3}{}{\parens[#1]{#3}}}}
\newcommand{\cube}[4][]{Q\ifblank{#2}{}{_{#2}}\ifblank{#4}{}{^{#4}}\ifblank{#3}{}{\parens[#1]{#3}}}

\newcommand{\restrfun}[1]{_{\vert {#1}}} 

\newcommand{\seq}[3][]{\parens[#1]{#2}_{#3}}

\DeclareMathOperator{\id}{id}

\DeclareMathOperator{\supp}{supp}

\DeclareMathOperator{\dist}{dist}

\DeclareMathOperator{\tr}{tr}

\DeclareMathOperator{\Int}{int}

\newcommand{\Rbase}{\mR}

\DeclarePairedDelimiterX{\abs}[1]{\lvert}{\rvert}{\ifblank{#1}{\:\cdot\:}{#1}}

\DeclarePairedDelimiterX{\norm}[1]{\lVert}{\rVert}{\ifblank{#1}{\:\cdot\:}{#1}}
\DeclarePairedDelimiter{\parens}{\lparen}{\rparen}

\DeclarePairedDelimiterX{\floor}[1]{\lfloor}{\rfloor}{\ifblank{#1}{\:\cdot\:}{#1}}

\DeclarePairedDelimiterX{\set}[2]{\lbrace}{\rbrace}{\ifblank{#2}{#1}{#1\st #2}}
\DeclarePairedDelimiterX{\ooInterval}[2]{\lparen}{\rparen}{#1,#2}
\DeclarePairedDelimiterX{\ccInterval}[2]{\lbrack}{\rbrack}{#1,#2}
\DeclarePairedDelimiterX{\coInterval}[2]{\lbrack}{\rparen}{#1,#2}
\DeclarePairedDelimiterX{\ocInterval}[2]{\lparen}{\rbrack}{#1,#2}
\DeclarePairedDelimiter{\tuple}{\lparen}{\rparen}
\DeclarePairedDelimiterX{\farg}[1]{\lparen}{\rparen}{\ifblank{#1}{\:\cdot\:}{#1}}
\DeclarePairedDelimiterX{\linfarg}[1]{\lbrack}{\rbrack}{\ifblank{#1}{\:\cdot\:}{#1}}
\DeclarePairedDelimiterX{\ScalarProd}[2]{\lparen}{\rparen}{\ifblank{#1}{\:\cdot\:}{#1} \mathbin\delimsize| \ifblank{#2}{\:\cdot\:}{{#2}}}
\DeclarePairedDelimiterX{\DualityProd}[2]{\langle}{\rangle}{\ifblank{#1}{\:\cdot\:}{#1},\ifblank{#2}{\:\cdot\:}{#2}}
\DeclarePairedDelimiterX{\FromTo}[2]{\lparen}{\rparen}{\ifblank{#2}{#1}{#1;#2}}

\numberwithin{equation}{section}

\begin{document}



\title{A surprising threshold for the validity of the method of singular projection}
\author{Antoine Detaille\footnote{Universite Claude Bernard Lyon 1, CNRS, Centrale Lyon, INSA Lyon, Université Jean Monnet, ICJ UMR5208, 69622 Villeurbanne, France.\newline
Departement Mathematik, ETH Zürich, Rämistrasse 101, 8092 Zürich, Switzerland\\
		E-mail:~\texttt{antoine.detaille@math.ethz.ch}}}


\maketitle

\begin{abstract}
	Given a compact manifold \( \mathcal{N} \) embedded into \( \varmathbb{R}^{\nu} \) and a projection \( P \) that retracts \( \varmathbb{R}^{\nu} \) except a singular set of codimension \( \ell \) onto \( \mathcal{N} \), we investigate the maximal range of parameters \( s \) and \( p \) such that the projection \( P \) can be used to turn an \( \varmathbb{R}^{\nu} \)-valued \( W^{s,p} \) map into an \( \mathcal{N} \)-valued \( W^{s,p} \) map.
	Devised by Hardt and Lin with roots in the work of Federer and Fleming, the method of projection is known to apply in \( W^{1,p} \) if and only if \( p < \ell \), and has been extended in some special cases to more general values of the regularity parameter \( s \).
	As a first result, we prove in full generality that, when \( s \geq 1 \), the method of projection can be applied in the whole expected range \( sp < \ell \).
	
	When \( 0 < s < 1 \), the method of projection was only known to be applicable when \( p < \ell \), a more stringent condition than \( sp  < \ell \).
	As a second result, we show that, somehow surprisingly, the condition \( p < \ell \) is optimal, by constructing, for every \( 0 < s < 1 \) and \( p  \geq \ell \), a bounded \( W^{s,p} \) map into \( \varmathbb{R}^{\ell} \) whose singular projections onto the sphere \( \varmathbb{S}^{\ell-1} \) all fail to belong to \( W^{s,p} \).
	
	As a byproduct of our method, a similar conclusion is obtained for the closely related method of almost retraction, devised by Haj\l asz, for which we also prove a more stringent threshold of applicability when \( 0 < s < 1 \).
\end{abstract}

\tableofcontents


\section{Introduction}
\label{sect:intro}

This paper is concerned with the implementation of the so-called \emph{method of singular projection} in the context of Sobolev spaces of mappings with values into a compact manifold.
More precisely, we let \( \mM \) and \( \mN \) be compact Riemannian manifolds, where \( \mM \) has dimension \( m \) and may have a boundary, and \( \mN \) is isometrically embedded into \( \bR^{\nu} \).
This latter assumption is without loss of generality by virtue of the Nash isometric embedding theorem~\cite{Nash1954,Nash1956}.
We work in the Sobolev space \( W^{s,p}\FromTo{\mM}{\mN} \) consisting of all maps \( u \in W^{s,p}\FromTo{\mM}{\bR^{\nu}} \) such that \( u\farg{x} \in \mN \) for almost every \( x \in \mM \).
This in particular comprises the case where the domain is a smooth bounded open subset \( \Omega \subset \bR^{m} \), by letting \( \mM = \overline{\Omega} \).

Although \( W^{s,p}\FromTo{\mM}{\mN} \) is a metric subspace of the linear Sobolev space \( W^{s,p}\FromTo{\mM}{\bR^{\nu}} \), the nonlinear constraint imposed by \( \mN \) prevents it from being a linear space.
Even more dramatically, most of the linear constructions that are usual when working with classical Sobolev spaces \emph{cannot} be implemented as such in the presence of the manifold constraint.

As a model example for our discussion, let us consider the \emph{extension of traces} problem.
It is well-known~\cite{Gagliardo1957} that every \( u \in W^{1-1/p,p}\FromTo{\mM}{} \) is the trace of a \( W^{1,p} \) function on \( \mM \times \ooInterval{0}{1} \), that is, \( u \) can be extended to a function \( U \in W^{1,p}\FromTo{\mM \times \ooInterval{0}{1}}{} \), such that \( \tr_{\mM} U = u \).
We denote by \( \ball{r}{a}{m} \) the open ball with center \( a \) and radius \( r \) in \( \bR^{m} \), and we use the notation \( \ball{}{}{m} = \ball{1}{0}{m} \) for the unit ball.
The usual construction of the extension \( U \) relies on a convolution product, and essentially consists in defining \( U\farg{x,t} = \varphi_{t} \ast u\farg{x} \) in the Euclidean setting, where \( \varphi \in C^{\infty}\cs\FromTo{\ball{}{}{m}}{} \) is a standard mollifying kernel --- with some technical adaptations to deal with points near the boundary.
In particular, this construction shows that the extension may even be taken to be \emph{smooth} in \( \mM \times \ooInterval{0}{1} \).
However, the convolution product being merely an averaging process, even if \( u \in \mN \) almost everywhere on \( \mM \), \( U \) \emph{need not} take its values on \( \mN \) almost everywhere on \( \mM \times \ooInterval{0}{1} \).
Therefore, one cannot deduce that every \( u \in W^{1-1/p,p}\FromTo{\mM}{\mN} \) is the trace of a mapping \( U \in W^{1,p}\FromTo{\mM \times \ooInterval{0}{1}}{\mN} \) as a direct consequence of the linear result, and not even of its usual proof.

If there would exist a smooth map \( \upPi \colon \bR^{\nu} \to \mN \) such that \( \upPi\restrfun{\mN} = \id_{\mN} \), then one could bring back to the linear case relying on the continuity of the composition operator, as \( \upPi \circ U \) would provide a \( W^{1,p} \) extension of \( u \) on \( \mM \times \ooInterval{0}{1} \).
However, there is no hope for such a globally defined smooth retraction onto \( \mN \) in the general case: if \( \mN \) is a compact manifold without boundary, then it can be shown that such a map \( \upPi \) \emph{never exists}.
However, in some special situations, one may circumvent this issue by working instead with a \emph{singular projection}.

Let us illustrate the key idea on the model case where \( \mN = \bS^{\ell-1} \subset \bR^{\ell} \), with \( \ell \geq 2 \), following the work of Hardt and Lin~\cite{HardtLin1987}.
In this case, one may consider the map \( P \colon \bR^{\ell} \setminus \set{0}{} \to \bS^{\ell-1} \) defined by \( P\farg{x} = \frac{x}{\abs{x}} \), which satisfies \( P\restrfun{\bS^{\ell-1}} = \id_{\bS^{\ell-1}} \).
The map \( P \) is a retraction, singular at the origin, satisfying the important estimate
\begin{equation}
\label{eq:e8d1b2cd1149a161}
	\abs{\Dext^{j}P\farg{x}}
	\leq
	C_{j}\frac{1}{\abs{x}^{j}}
	\quad
	\text{for every \( x \in \bR^{\ell} \setminus \set{0}{} \),}
\end{equation}
where \( C_{j} > 0 \) is a constant depending only on \( j \).
A natural idea would be to conclude by considering the map \( P \circ U \).
However, this attempt faces two important issues:
\begin{enumerate*}[label=(\(\textup{\roman*}\))]
	\item it may happen that \( U \) takes the value \( 0 \) on a large set, in which case \( P \circ U \) would not be defined on this set;
	\item even if \( U \neq 0 \) almost everywhere, so that \( P \circ U \) is well-defined almost everywhere, one cannot apply standard results about the continuity of the composition operator to deduce that \( P \circ U \in W^{1,p} \), because of the singularity of \( P \) at \( 0 \).
\end{enumerate*}

To overcome these difficulties, one relies on the following ingenious \emph{averaging argument}, which takes its roots in the seminal work by Federer and Fleming~\cite{FedererFleming1960}.
Given \( a \in \ball{}{}{\ell} \), consider the map \( P \circ \parens{U-a} \).
In what follows, we assume for simplicity that \( \mM = \cball{}{}{\ell-1} \).
By Sard's lemma, for almost every \( a \in \ball{}{}{\ell} \), \( U^{-1}\farg{\set{a}{}} \) is a finite union of points in \( \ball{}{}{\ell-1} \times \ooInterval{0}{1} \), so that \( P \circ \parens{U-a} \) is defined almost everywhere.
Let us give an estimate of its weak derivative.
By the chain rule, we know that
\begin{equation}
\label{eq:f57bbd49ae63f6ed}
	\abs{\Dext{\parens{P \circ \parens{U-a}}}\farg{x,t}}
	\lesssim
	\frac{\Dext{U}\farg{x,t}}{\abs{U\farg{x,t}-a}}\text{.}
\end{equation}
The key idea behind the averaging argument is to first integrate with respect to \( a \): if \( p < \ell \), one has
\begin{multline}
\label{eq:480d8d2a5d5cd29c}
	\int_{\ball{}{}{\ell}} \abs{\Dext{\parens{P \circ \parens{U-a}}}\farg{x}}^{p} \intd{a}
	\lesssim
	\abs{\Dext{U}\farg{x,t}}^{p}\int_{\ball{}{}{\ell}}\frac{1}{\abs{U\farg{x,t}-a}^{p}}\intd{a}\\
	\leq
	\abs{\Dext{U}\farg{x,t}}^{p}\int_{\ball{2}{}{\ell}} \frac{1}{\abs{a}^{p}}\intd{a}
	\lesssim
	\abs{\Dext{U}\farg{x,t}}^{p}
	\text{,}
\end{multline}
where we have used the fact that \( \abs{U\farg{x,t}} \leq 1 \) for every \( \tuple{x,t} \in \ball{}{}{\ell-1} \times \ooInterval{0}{1} \).
Integrating in \( \tuple{x,t} \) and using Tonelli's theorem, we find
\[
	\int_{\ball{}{}{\ell}} \abs{\Dext{\parens{P \circ \parens{U-a}}}}_{L^{p}\FromTo{\ball{}{}{\ell-1}\times \ooInterval{0}{1}}{}}^{p} \intd{a}
	\lesssim
	\abs{\Dext{U}}_{L^{p}\FromTo{\ball{}{}{\ell-1}\times\ooInterval{0}{1}}{}}^{p}\text{.}
\]
This implies the existence of a measurable set \( A \subset \ball{}{}{\ell} \) of positive measure such that, for every \( a \in A \),
\[
	\abs{\Dext{\parens{P \circ \parens{U-a}}}}_{L^{p}\FromTo{\ball{}{}{\ell-1}\times \ooInterval{0}{1}}{}}^{p}
	\lesssim
	\abs{\Dext{U}}_{L^{p}\FromTo{\ball{}{}{\ell-1}\times\ooInterval{0}{1}}{}}^{p}\text{.}
\]
One concludes by observing that \( P\parens{\cdot-a}\restrfun{\bS^{\ell-1}} \) is a smooth diffeomorphism on \( \bS^{\ell-1} \), and hence \( \parens{P\parens{\cdot-a}^{-1}}\restrfun{\bS^{\ell-1}} \circ \parens{P \circ \parens{U-a}} \) is a \( W^{1,p}\FromTo{\ball{}{}{\ell-1} \times \ooInterval{0}{1}}{\bS^{\ell-1}} \) map whose trace on \( \ball{}{}{\ell} \) coincides with \( u \).

\bigskip

In the above example, the important features of the singular projection \( x \mapsto \frac{x}{\abs{x}} \) onto \( \bS^{\ell-1} \) were that it is smooth outside of a small singular set, and that the rate of blow-up of its derivatives when approaching the singular set is suitably controlled.
This motivates the following more general definition of a singular projection.

\begin{definition}
\label{def:singular_projection}
	Let \( \ell \in \set{2,\dotsc,\nu}{} \).
	An \emph{\( \ell \)\=/singular projection onto \( \mN \)} is a smooth map \( P \colon \bR^{\nu} \setminus \Sigma \to \mN \)
	such that \( P_{\vert \mN} = \id_{\mN} \) and 
	\[
	\abs{\Dext^{j}P\parens{x}} \leq C_{j}\frac{1}{\dist{\parens{x,\Sigma}}^{j}}
	\quad
	\text{for every \( x \in \bR^{\nu} \setminus \Sigma \) and \( j \in \bN_{\ast} \)}
	\]
	for some constant \( C_{j} > 0 \) depending on \( j \),
	where \( \Sigma \subset \bR^{\nu} \setminus \mN \) is  a union of closedly embedded \( \parens{\nu-\ell} \)-submanifolds of \( \bR^{\nu} \).
\end{definition}

The reader may wonder which target manifolds \( \mN \) do admit an \( \ell \)\=/singular projection for some \( \ell \).
A necessary an sufficient condition for the existence of a singular projection, depending only on the topology of \( \mN \), is given by Proposition~\ref{prop:charact_man_sing_proj}.
Other important properties of singular projections shall be recalled in Section~\ref{sect:singular_projections}.
For the moment, let us simply note that the map \( P \colon \bR^{\ell} \setminus \set{0}{} \to \bS^{\ell-1} \), \( P\farg{x} = \frac{x}{\abs{x}} \) defined above is indeed an \( \ell \)\=/singular projection, with singular set \( \Sigma =  \set{0}{} \).

The illustrative argument presented above can be repeated mutatis mutandis for a general singular projection, hence showing that the method of singular projection can be implemented in full generality in \( W^{1,p} \), as soon as \( \mN \) admits an \( \ell \)\=/singular projection and \( p < \ell \).
The argument may also easily be adapted to cover higher order spaces \( W^{k,p} \).
In this setting, the main difference is that the right-hand side of~\eqref{eq:f57bbd49ae63f6ed}, obtained via the chain rule, will involve a term of the form
\(
	\frac{1}{\dist{\farg{U\farg{x,t}-a,\Sigma}}^{k}}\text{,}
\)
so that, to ensure the finiteness of the last integral in the corresponding version of~\eqref{eq:480d8d2a5d5cd29c}, the assumption on the parameters will be \( kp < \ell \).
Similarly, for fractional order spaces, the expected assumption on the parameters to be satisfied to ensure the applicability of the method of singular projection is \( sp < \ell \).
Our first main result shows that this is indeed the case in the range \( s \geq 1 \).

\begin{theorem}
	\label{theorem:estimate_sgeq1}
	Let \( P \colon \bR^{\nu} \setminus \Sigma \to \mN \) be an \( \ell \)-singular projection.
	If \( s \geq 1 \) and \( sp < \ell \), then for every map \( u \in W^{s,p}\FromTo{\mM}{\bR^{\nu}} \cap L^{\infty}\FromTo{\mM}{\bR^{\nu}} \) and every \( \alpha > 0 \), we have
	\[
		\int_{\ball{\alpha}{}{\nu}} \abs{P \circ (u-a)}_{W^{s,p}\FromTo{\mM}{}}^{p}\intd{a}
		\leq
		C\norm{u}_{W^{s,p}\FromTo{\mM}{}}^{p}\text{,}
	\]
	for some constant \( C > 0 \) depending on \( s \), \( p \), \( \mM \), \( P \), \( \alpha \), and \( \norm{u}_{L^{\infty}\FromTo{\mM}{}} \).
\end{theorem}

We observe that the assumption \( sp < \ell \) cannot be relaxed.
Indeed, in the model situation where \( \mM = \cball{}{}{\ell} \), \( \mN = \bS^{\ell-1} \), and \( P\farg{x} = \frac{x}{\abs{x}} \), if one considers the map \( u\farg{x} = x \), then it is easily seen that \( P \circ \parens{u-a} \) belongs to \( W^{s,p}\FromTo{\ball{}{}{\ell}}{} \) for \emph{no} \( a \in \ball{}{}{\ell} \) in case \( sp \geq \ell \).

When \( s \geq 1 \), Theorem~\ref{theorem:estimate_sgeq1} establishes the applicability of the method of singular projection in the full expected range.
This result is in the spirit of classical results concerning \emph{superposition operators}, see e.g.~\cite{BrezisMironescu2001} and the references therein, the main difference here being that the map with which we compose is allowed to have singularities.

The range \( 0 < s < 1 \), on the contrary, is source of additional difficulties.
Indeed, it is well-known to experts, and can easily be shown, for instance, by adapting the proof of Theorem~\ref{theorem:estimate_sgeq1}, that the method of singular projection can be implemented under the assumption that \( p < \ell \); see e.g.~\cite[Proof of Theorem~2.11]{Vincent2025} for a similar argument.
This assumption is \emph{more stringent} than the expected assumption \( sp < \ell \).

Two hints supporting the conjecture that the method of singular projection can be implemented whenever \( sp < \ell \) also when \( 0 < s < 1 \) are the following.
\begin{enumerate*}[label=(\(\textup{\roman*}\))]
	\item In many other problems related to Sobolev mappings to manifolds, the product \( sp \) is the right quantity to consider in the full range \( 0 < s < +\infty \).
	A typical example is given by the local obstruction to the strong density of smooth maps in \( W^{s,p}\FromTo{\mM}{\mN} \), which comes from the nontriviality of \( \pi_{\floor{p}}\FromTo{\mN}{} \) when \( s = 1 \)~\cite{SchoenUhlenbeck1983,Bethuel1991,HangLin2003II}, and from the nontriviality of \( \pi_{\floor{sp}}\FromTo{\mN}{} \) when \( 0 < s < 1 \)~\cite{BrezisMironescu2015}.
	\item If \( u \in W^{s,p}\FromTo{\ball{}{}{\ell}}{\ball{}{}{\ell}} \) is given by \( u\farg{x} = x \), which is the model example to show that the condition \( sp < \ell \) cannot be improved, then \( P \circ \parens{u-a} = \frac{u-a}{\abs{u-a}} \in W^{s,p}\FromTo{\ball{}{}{\ell}}{\bS^{\ell-1}} \) whenever \( sp < \ell \), and the corresponding estimate holds.
	This can be shown either by interpolation, see e.g.~\cite[Lemma~15.12]{BrezisMironescu2021}, or by a direct argument, see e.g.~\cite[Lemma~3.7]{VanSchaftingenOxfordNotes}.
	Although this is a very special Sobolev map, it could nevertheless be thought of as somehow representative of the general case, as any smooth map essentially looks like the identity locally on every affine space normal to one of its level sets.
\end{enumerate*}

Our second main result shows that, somehow surprisingly, the assumption \( p < \ell \) is \emph{optimal} in the range \( 0 < s < 1 \) to ensure the applicability of the method of singular projection, already in the model case of the projection \( P \colon \bR^{\ell} \setminus \set{0}{} \to \bS^{\ell-1} \) given by \( P\farg{x} = \frac{x}{\abs{x}} \).

\begin{theorem}
	\label{theorem:ctrex_sle1}
	Assume that \( 0 < s < 1 \) and \( 1 \leq p < +\infty \) are such that \( sp < \ell \) but \( p \geq \ell \), and let \( P \colon \bR^{\ell} \setminus \set{0}{} \to \bS^{\ell-1} \) be given by \( P\farg{x} = \frac{x}{\abs{x}} \).
	There exists a map \( u \in W^{s,p}\FromTo{\ball{}{}{\ell}}{\bR^{\ell}} \cap L^{\infty}\FromTo{\ball{}{}{\ell}}{\bR^{\ell}} \) such that \( P \circ \parens{u-a} \notin W^{s,p} \) for every \( a \in \ball{}{}{\ell} \).
\end{theorem}

The proof of Theorem~\ref{theorem:ctrex_sle1} shall be given in Section~\ref{sect:0lesle1}.
However, let us already briefly give an intuition about why the \( W^{s,p} \) seminorm with \( s < 1 \) is sensitive to the threshold \( p < \ell \) rather than \( sp < \ell \).
For this purpose, let us consider the composition of a \( W^{s,p} \) map \( u \) with a (globally defined, without singular set) map \( F \), with Lipschitz constant equal to \( L > 0 \).
A straightforward estimate on the numerator appearing in the Gagliardo seminorm shows that
\[
	\abs{F \circ u}_{W^{s,p}}^{p}
	\leq
	L^{p}\abs{u}_{W^{s,p}}^{p}\text{.}
\]
Moreover, this estimate is optimal, as can be seen by taking \( F \) to be an affine map.
Therefore, the power that appears in this estimate for the composition with a Lipschitz map is indeed \( p \), and not \( sp \).

This observation will be at the heart of our proof of Theorem~\ref{theorem:ctrex_sle1}.
Namely, we will start from the fact that, if a map \( u \) takes two values \( c^{+} \), \( c^{-} \in \ball{}{}{\ell} \) on two balls of fixed radius, and if \( c^{+} \) and \( c^{-} \) are close to the projection point \( a \), then this results in a multiplication of the contribution of the associated region to the Gagliardo seminorm by a factor \( \dist{\parens{c^{+},c^{-}}}^{-p} \).
Gluing together several such maps, for some suitably chosen pairs of points \( c^{+} \) and \( c^{-} \), one is then able to construct a \( W^{s,p} \) map whose energy will be dramatically increased after composition with a singular projection \emph{for any} choice of the point \( a \).
Making such a construction in a careful way and providing the associated estimates on how the singular projection increases the Sobolev energy then allows to conclude via a \emph{nonlinear uniform boundedness principle}; see~\cite{MonteilVanSchaftingen2019}.
We note that, as a byproduct of the use of this nonlinear uniform boundedness principle, we actually obtain that maps \( u \) as in the conclusion of Theorem~\ref{theorem:ctrex_sle1} are even dense in \( W^{s,p}\FromTo{\ball{}{}{\ell}}{\bR^{\ell}} \).

This idea can also be adapted to prove analogously a more stringent condition than expected for the applicability of another closely related method, namely the \emph{method of almost retraction}.
In order to keep this introduction focused on the method of singular projection, we postpone further explanations on this matter to the end of this text, in Section~\ref{sect:almost_projection}.

We conclude this introduction with a brief survey of known applications of the method of projection, to motivate its study.

\paragraph{Extension of traces.}

As we mentioned in our introductory example, the first use of the method of singular projection in the context of mappings to manifolds goes back to the work by Hardt and Lin~\cite{HardtLin1987}, where it was used to obtain extension of traces results preserving the manifold constraint.
In Section~\ref{sect:sgeq1}, we will show that, when \( s+\frac{1}{p} \geq 1 \), the existence of an \( \ell \)\=/singular projection onto \( \mN \) always implies that every map in \( W^{s,p}\FromTo{\partial\mM}{\mN} \) is the trace of a map in \( W^{s+\frac{1}{p},p}\FromTo{\mM}{\mN} \) as soon as \( sp+1 < \ell \) and \( s \) is not an integer; see Theorem~\ref{theorem:extension}.

Although not coming as a surprise, this result is, to the best of our knowledge, the first extension result in higher order Sobolev spaces of mappings, that is, when \( s > 1-\frac{1}{p} \).

In view of Proposition~\ref{prop:charact_man_sing_proj}, this approach applies to \( \floor{sp} \)\=/connected targets \( \mN \).
When \( s = 1-\frac{1}{p} \) --- and thus \( \floor{sp} = \floor{p-1} \) --- an important amount of research has been devoted to the study of the extension of traces problem for mappings into manifolds.
Hardt and Lin~\cite{HardtLin1987}, and Bethuel and Demengel~\cite{BethuelDemengel1995} have shown that the vanishing of \( \pi_{\floor{p-1}}\FromTo{\mN}{} \) is \emph{necessary} to ensure that the extension of traces holds, due to the presence of a topological obstruction.
It has also been shown by Bethuel~\cite{Bethuel2014} that, if \( \pi_{j}\FromTo{\mN}{} \) is \emph{infinite} for some \( 1 \leq j \leq \floor{p-1} \), then an analytical type obstruction to the extension of traces arises; see also the work of Mironescu and Van Schaftingen~\cite{MironescuVanSchaftingen2021Toul}.
This left open the case where \( \pi_{\floor{p-1}}\FromTo{\mN}{} = \set{0}{} \) and all \( \pi_{j}\FromTo{\mN}{} \) with \( 1 \leq j < \floor{p-1} \) are finite but at least one of them is nontrivial.
This missing case was solved recently by Van Schaftingen in a groundbreaking contribution~\cite{VanSchaftingen2024}, which showed that the two above mentioned obstructions are the \emph{only ones}: if \( \pi_{\floor{p-1}}\FromTo{\mN}{} = \set{0}{} \) and all \( \pi_{j}\FromTo{\mN}{} \) with \( 1 \leq j < \floor{p-1} \) are finite, then the extension of traces holds when \( \mM = \cball{}{}{m} \).
The case where \( \mM \) has nontrivial topology is more involved, as global topological obstructions may arise, and we refer the reader to~\cite{VanSchaftingen2024} for more details on this matter.

This raises the following open problem.
We state it for the extension from \( \mM \) to \( \mM \times \ooInterval{0}{1} \) to avoid having to deal with global topological obstructions that arise in the general case of an extension from \( \partial\mM \) to \( \mM \), which is outside of the scope of this short motivating example, and refer to \cite{VanSchaftingen2024} and the references therein for more details on this issue.

\begin{openproblem}
	Is it true that every \( u \in W^{s,p}\FromTo{\mM}{\mN} \) is the trace on \( \mM \times \set{0}{} \) of a map \( U \in W^{s+\frac{1}{p},p}\FromTo{\mM \times \ooInterval{0}{1}}{\mN} \) if and only if \( \pi_{\floor{sp}}\FromTo{\mN}{} = \set{0}{} \) and all \( \pi_{j}\FromTo{\mN}{} \) with \( 1 \leq j < \floor{sp} \) are finite in the whole range \( 0 < s < +\infty \) when \( s \) is not an integer?
\end{openproblem}

For topological reasons, the method of projection can only allow to settle the more restrictive case where \( \mN \) is \( \floor{sp} \)\=/connected.
We explain how to implement this approach in Theorem~\ref{theorem:extension} below.
Moreover, the case where \( s < 1-\frac{1}{p} \) remains completely open, even under this more restrictive assumption, due to the failure of the method of projection given by Theorem~\ref{theorem:ctrex_sle1}.

\paragraph{Strong density of almost smooth maps}

Another widespread use of the method of singular projection is related to the so-called \emph{strong density problem}.
A well-known striking fact when studying Sobolev spaces of mappings is that \( C^{\infty}\FromTo{\mM}{\mN} \) \emph{need not} be dense in \( W^{s,p}\FromTo{\mM}{\mN} \)~\cite{SchoenUhlenbeck1983}.
However, it has been established by Bethuel in the seminal contribution~\cite{Bethuel1991} in the case \( s = 1 \) --- see also the work of Hang and Lin \cite{HangLin2003II} --- and later on extended to the full range \( 0 < s < +\infty \)~\cite{BousquetPonceVanSchaftingen2015,BrezisMironescu2015,Detaille2023}, that one obtains a dense class in \( W^{s,p}\FromTo{\mM}{\mN} \) by considering instead the class of mappings \( u \in W^{s,p}\FromTo{\mM}{\mN} \) that are smooth outside of a finite union of \( \parens{m-\floor{sp}-1} \)\=/dimensional submanifolds of \( \mM \).

The proof of this fact is technically involved, but in the special situation where the target manifold \( \mN \) admits an \( \ell \)\=/singular projection with \( sp < \ell \), the proof can be considerably simplified by relying on the method of singular projection.
Such an idea was already present in the earlier work by Bethuel and Zheng~\cite{BethuelZheng1988}, and was later on implemented by Escobedo~\cite{Escobedo1988}, Rivière~\cite{Riviere2000}, Bourgain, Brezis, and Mironescu~\cite{BourgainBrezisMironescu2005}, Bousquet~\cite{Bousquet2007}, and Bousquet, Ponce, and Van Schaftingen~\cite{BousquetPonceVanSchaftingen2014}.
For closely related directions of research, see also the work by Haj\l asz~\cite{Hajlasz1994} \( s = 1 \), and its extension to \( s \geq 1 \)~\cite{BousquetPonceVanSchaftingen2013}, where a method of \emph{almost projection} was implemented; see Section~\ref{sect:almost_projection} for more details and a study of the case \( 0 < s < 1 \).
The full range of applicability of the method of singular projection in the study of the strong density problem was obtained in~\cite{Detaille2025}.

Theorem~\ref{theorem:estimate_sgeq1} dealing with the projection of \emph{one fixed} map \( u \in W^{s,p}\FromTo{\mM}{\mN} \), it cannot be used as such to obtain strong density results for almost smooth maps.
However, as we shall see in Section~\ref{sect:sgeq1}, the \emph{proof} of Theorem~\ref{theorem:estimate_sgeq1} can be readily adapted to yield a variant suited for sequences; see Theorem~\ref{theorem:conv_sgeq1}.
Using this variant, we shall revisit~\cite[Theorem~2.4]{Detaille2025} in the range \( s \geq 1 \); see Corollary~\ref{corollary:density_class_R}.

\bigskip

In the two previous problems, the method of projection was applied to functions that arise from mappings taking values into the target manifold \( \mN \), for instance via a convolution procedure.
However, there are some instances where one might be willing to project a more general function with values into the ambient space \( \bR^{\nu} \).
We briefly explain, with very few details, some such instances.

\paragraph{Weak density of smooth maps}

In~\cite{PakzadRiviere2003}, Pakzad and Rivière proved that, for a large class of target manifolds \( \mN \), any \( u \in W^{1,2}\FromTo{\mM}{\mN} \) can be \emph{weakly} approximated by smooth mappings.
The model situation in their work is the case where \( \mN \) is simply connected --- we note that, in this special situation, the weak density of smooth maps was already known via the \emph{method of almost projection} introduced by Haj\l asz~\cite{Hajlasz1994}, and over which we shall come back in Section~\ref{sect:almost_projection}.
In this case, a crucial step in their argument is to be able to project the map \( u \) to the \( 2 \)\=/dimensional skeleton \( \mN^{2} \) of \( \mN \), obtaining a map \( u^{2} \in W^{1,2}\FromTo{\mM}{\mN^{2}} \).
Then, they could rely on the fact that, for a simply connected compact manifold \( \mN \), the skeleton \( \mN^{2} \) has a quite simple topology: it is homotopically equivalent to a finite bouquet of spheres.

In this situation, similarly to the one studied here, there exists a map \( P \) that retracts almost all \( \mN \), with a singular set of codimension \( 3 \), to the \( 2 \)\=/skeleton \( \mN^{2} \).
The aforementioned map \( u^{2} \) is constructed by applying the counterpart of Theorem~\ref{theorem:estimate_sgeq1} in this setting (with \( s = 1 \)) to the map \( u \) with the singular retraction \( P \).

In view of Theorem~\ref{theorem:estimate_sgeq1}, there is hope that the approach by Pakzad and Rivière can be extended to obtain weak approximation by smooth maps in \( W^{s,p}\FromTo{\mM}{\mN} \) when \( sp = 2 \), \( s \geq 1 \), and \( \mN \) is, for instance, simply connected --- although we shall not attempt to carry on this task in the present work, as it would require a whole paper on its own.
However, Theorem~\ref{theorem:ctrex_sle1} suggests that, in the range \( 0 < s < 1 \), the approach would require substantial modifications.

To conclude this part with a precise question, we formulate the following open problem.

\begin{openproblem}
	Let \( 0 < s < 1 \) and \( 1 \leq p < +\infty \) be such that \( sp \in \bN_{\ast} \), and assume that \( \mN \) is \( \parens{sp-1} \)\=/connected.
	Is it true that every map \( u \in W^{s,p}\FromTo{\mM}{\mN} \) can be approximated with a sequence of smooth maps from \( \mM \) to \( \mN \)?
\end{openproblem}

The counterpart of this question for \( s \geq 1 \) is known to have a positive answer, obtained via the method of almost projection introduced by Haj\l asz \cite{Hajlasz1994} when \( s = 1 \), and pursed by Bousquet, Ponce, and Van Schaftingen for \( s \geq 1 \) \cite{BousquetPonceVanSchaftingen2013}.
The case \( 0 < s < 1 \) however faces similar difficulties as for the method of singular projection when attempting to implement this method of almost retraction.
We shall give more details on this matter in Section~\ref{sect:almost_projection}.

\paragraph{The singular set of Sobolev mappings}

As we already explained, even though smooth mappings are in general not dense in the space \( W^{s,p}\FromTo{\mM}{\mN} \), one nevertheless has at hand a convenient dense class provided by those \( W^{s,p}\FromTo{\mM}{\mN} \) that are smooth outside of a singular set of dimension \( m-\floor{sp}-1 \).
A question that has received a lot of interest is whether this singular set can be given a robust meaning so that it passes to the limit, allowing one to define the singular set of \emph{any} \( W^{s,p}\FromTo{\mM}{\mN} \) map.
The hope behind this endeavor is that the singular set \( S_{u} \) of a given map \( u \in W^{s,p}\FromTo{\mM}{\mN} \) would then encode the obstruction to approximate \( u \) by smooth maps, and be so that \( S_{u} = 0 \) if and only if \( u \) is a strong limit of smooth mappings.

Several directions of research have been pursued to attempt tackling this problem, mainly in the model case \( s = 1 \).
We notably mention the study of the \emph{Jacobian}, see e.g.~\cite{BrezisCoronLieb1986,Bethuel1990,BethuelCoronDemengelHelein1991,AlbertiBaldoOrlandi2003,BourgainBrezisMironescu2005,Bousquet2007,BousquetMironescu2014,Mucci2024}, the development of \emph{Cartesian currents} by Giaquinta, Modica, Sou\v{c}ek, and collaborators, culminating at the monograph~\cite{GiaquintaModicaSoucek1998,GiaquintaModicaSoucek1998II}, and the introduction of the notion of \emph{scans} by Hardt and Rivière~\cite{HardtRiviere2003,HardtRiviere2008}.

In~\cite{PakzadRiviere2003}, Pakzad and Rivière suggested to view the singular set of a Sobolev map as a \emph{flat chain} with values into the group \( \pi_{\floor{p}}\FromTo{\mN}{} \), using the language of \emph{geometric measure theory}.
More specifically, they were able to define a robust notion of singular set for any \( \mN \)\=/valued \( W^{1,p} \) map whenever the target is \( \parens{\floor{p}-1} \)\=/connected, under the extra assumption that either \( \floor{p} = 1 \) or \( \floor{p} = m-1 \), which was required because of their use of a result from geometric measure theory which is not available in full generality.

In~\cite{CanevariOrlandi2019}, Canevari and Orlandi removed this extra assumption and handled the case where \( \mN \) is \( \parens{\floor{p}-1} \)\=/connected in full generality.
The key idea at the core of their reasoning involves defining the singular set of \emph{any} map \( u \in W^{1,p}\FromTo{\mM}{\bR^{\nu}} \), \emph{non-necessarily \( \mN \)\=/valued}, by retracting \( u \) to \( \mN \) using the method of singular projection.
In particular, if \( u \) is already \( \mN \)\=/valued, their notion of singular set coincides with the one defined by Pakzad and Rivière in~\cite{PakzadRiviere2003}.

Therefore, this features another instance where it is crucial to be able to apply the method of singular projection to \emph{any} \( \bR^{\nu} \)\=/valued Sobolev map, not necessarily obtained from an \( \mN \)\=/valued map for example by convolution.

\paragraph{Regularity of harmonic maps}

Among the research fields which are tightly connected to the study of Sobolev mappings to manifolds and which make intensive use of the results and techniques that the latter provides, is the study of \emph{harmonic mappings}.
Such mappings are defined as minimizers of the Dirichlet energy among all mappings taking their values into a given target manifold.
Challenging questions in this area of research pertain to the uniqueness, regularity, or stability of the minimizers.
Here also, the method of projection has been used several times in a crucial way, notably to provide suitable competitors for the Dirichlet energy.

The typical situation is the following.
One first constructs a map, for instance by performing the harmonic extension of a manifold-valued map defined on the boundary of some ball, or by gluing together two manifold-valued maps via a cutoff procedure.
The map obtained this way need not be a competitor for a Dirichlet problem under constraint, as it need not take its values into the prescribed target.
One then has to correct this issue by projecting the aforementioned map back to the manifold through a singular projection, hence obtaining an admissible competitor.

Such an approach was implemented for instance by Hardt and Lin~\cite{HardtLin1987}, using a gluing by cutoff construction to obtain compactness of minimizers, or by Hardt, Kinderlehrer, and Lin~\cite{HardtKinderlehrerLin1986}, who obtained energy estimates for minimizers on balls relying on competitors built as projections of harmonic extensions from the boundary.
We also refer the reader to~\cite{MazowieckaMiskiewiczSchikorra2018} for a survey of such problems, and to~\cite{Gastel2016} for results in higher order spaces and~\cite{Vincent2025} for results in fractional Sobolev spaces.

\subsection*{Acknowledgements}


I am deeply grateful to Katarzyna Mazowiecka, Petru Mironescu, Augusto Ponce, and Akshara Vincent for many interesting discussions and comments all along the preparation of this paper, with special thanks to Katarzyna Mazowiecka for raising my attention to this question.
I thank the anonymous referee for many helpful remarks to improve the presentation and pedagogy of the text.

\section{Singular projections}
\label{sect:singular_projections}

In this short section, we collect some properties of singular projections that will be useful for us in the sequel.
We first recall the following characterization of those manifolds \( \mN \) that admit an \( \ell \)\=/singular projection.

\begin{proposition}
\label{prop:charact_man_sing_proj}
	The manifold \( \mN \) admits an \( \ell \)\=/singular projection if and only if it is \( \parens{\ell-2} \)\=/connected.
\end{proposition}

We refer the reader e.g.\ to~\cite[Lemmas~2.2 and~2.3]{Detaille2025} for a proof of this result.

We also recall the following lemma, which is a consequence of Sard's theorem and the submersion theorem, and which provides the structure of the singular set of the projection of a smooth map.

\begin{lemma}
\label{lemma:Sard}
	Let \( u \in C^{\infty}\FromTo{\mM}{\bR^{\nu}} \) and let \( \Sigma \) be a finite union of \( \parens{\nu-\ell} \)\=/submanifolds of \( \bR^{\nu} \).
	For almost every \( a \in \bR^{\nu} \), the set \( v^{-1}\farg{\Sigma+a} \) is a finite union of \( \parens{m-\ell} \)\=/submanifolds of \( \mM \), one for each submanifold constituting \( \Sigma \) --- or the empty set if \( \ell > m \).
\end{lemma}

We refer the reader e.g.\ to~\cite[Lemma~2.7(i)]{Detaille2025} for a proof of this lemma.
The usefulness of Lemma~\ref{lemma:Sard} in our work is twofold.
First, we shall combine it with Lemma~\ref{lemma:small_sing_set_Sobolev} below to deduce the Sobolev regularity of the projection we construct in the case where the projected map is smooth.

\begin{lemma}
\label{lemma:small_sing_set_Sobolev}
	Let \( S \subset \mM \) be a closed set such that \( \mH^{m-1}\farg{S} = 0 \), where \( \mH^{m-1} \) denotes the \( \parens{m-1} \)\=/dimensional Hausdorff measure.
	If \( u \in C^{\infty}\FromTo{\mM \setminus S}{} \) is such that \( u \in L^{p}\FromTo{\mM}{} \) and \( \Dext{u} \in L^{p}\FromTo{\mM}{} \), then \( u \in W^{1,p}\FromTo{\mM}{}\).
\end{lemma}

We refer the reader e.g.\ to~\cite[Proposition~4.18]{Ponce2016}, or to~\cite[Lemma~1.10]{BrezisMironescu2021} and the references therein.

Second, Lemma~\ref{lemma:Sard} ensures that, when the projected map is smooth, then its projection belongs to a suitable class of almost smooth maps.
To state this more precisely, we recall the definition of the class \( \Rbase_{s,p}\FromTo{\mM}{\mN} \) as the class of those maps \( u \in W^{s,p}\FromTo{\mM}{\mN} \) that are smooth outside of a finite union of \( \parens{m-\floor{sp}-1} \)\=/dimensional submanifolds of \( \mM \).
Then, Lemma~\ref{lemma:Sard} ensures that, if \( u \) is smooth, then \( P \parens{u-a} \) belongs to the class \( \Rbase_{s,p} \) for almost every \( a \) as soon as it belongs to \( W^{s,p} \), provided that \( \ell = \floor{sp}+1 \).

We conclude this section with a last result ensuring that, for \( a \) sufficiently small, \( P \parens{\cdot-a} \) restricts to a smooth diffeomorphism on \( \mN \).
This property is crucial if we want our projection to preserve the values of the projected map which are already on the manifold.
More specifically, the map that we will be willing to consider is the map \( \parens{P \parens{\cdot-a}\restrfun{\mN}}^{-1} \circ P \circ \parens{u-a} \), which, unlike \( P \circ \parens{u-a} \), coincides with \( u \) wherever \( u \) takes its values into \( \mN \).

\begin{proposition}
\label{prop:Inv_trans_proj}
	Let \( P \colon \bR^{\nu} \setminus \Sigma \to \mN \) be an \( \ell \)\=/singular projection.
	There exists \( \varepsilon > 0 \) such that, for every \( a \in \bR^{\nu} \) satisfying \( \abs{a} \leq \varepsilon \), the map \( P\parens{\cdot-a}\restrfun{\mN} \colon \mN \to \mN \) is a smooth diffeomorphism.
\end{proposition}
\begin{proof}
	This is a direct consequence of the fact that diffeomorphisms form an open set in the \( C^{1} \) topology, see e.g.~\cite[Theorem~1.6]{Hirsch1976}, as \( P\parens{\cdot-a}\restrfun{\mN} \) is close to \( \id_{\mN} \) in the \( C^{1} \) topology whenever \( \abs{a} \) is sufficiently small.
\end{proof}

\section{The case \texorpdfstring{\( s \geq 1 \)}{s <= 1}}
\label{sect:sgeq1}

This section is devoted to the proof of our positive result in the range \( s \geq 1 \), namely Theorem~\ref{theorem:estimate_sgeq1}.
For this purpose, we shall need the following lemma for estimating the Gagliardo seminorm of a product.

\begin{lemma}
\label{lemma:Gagliardo_seminorm_product}
	Let \( i \in \bN_{\ast} \), \( 1 \leq p < +\infty \), and \( 0 < \sigma < 1 \).
	If for every \( \alpha \in \set{1,\dotsc,i}{} \), \( v_{\alpha} \in L^{q_{\alpha}}\FromTo{\bR^{m}}{} \) and \( \Dext{v_{\alpha}} \in L^{r_{\alpha}}\FromTo{\bR^{m}}{} \), where \( 1 < r_{\alpha} < q_{\alpha} \) and
	\[
		\frac{1-\sigma}{q_{\alpha}} + \frac{\sigma}{r_{\alpha}} + \sum_{\substack{\beta=1 \\ \beta \neq \alpha}}^{i}\frac{1}{q_{\beta}} = \frac{1}{p}\text{,}
	\]
	then
	\[
		\abs[\bigg]{\prod_{\alpha=1}^{i}v_{\alpha}}_{W^{\sigma,p}\FromTo{\bR^{m}}{}}
		\leq
		C\sum_{\alpha=1}^{i}\parens[\bigg]{\norm{v_{\alpha}}_{L^{q_{\alpha}}\FromTo{\bR^{m}}{}}^{1-\sigma}\norm{\Dext{v_{\alpha}}}_{L^{r_{\alpha}}\FromTo{\bR^{m}}{}}^{\sigma}\prod_{\substack{\beta=1 \\ \beta \neq \alpha}}^{i}\norm{v_{\beta}}_{L^{q_{\beta}}\FromTo{\bR^{m}}{}}}\text{,}
	\]
	for some constant \( C > 0 \) depending on \( m \), \( \sigma \), \( p \), the \( r_{\alpha} \), and the \( q_{\alpha} \).
\end{lemma}

The above statement is taken from~\cite[Lemma~2.9]{BousquetPonceVanSchaftingen2013}, but the lemma was already proved implicitly in~\cite[Section~2]{MazyaShaposhnikova2002JEvolEqu}.

We now turn to the proof of Theorem~\ref{theorem:estimate_sgeq1}.
We mention importantly that our argument owes much to the proof of the continuity of the composition operator proposed by Maz'ya and Shaposhnikova~\cite{MazyaShaposhnikova2002JEvolEqu}, see also~\cite{BousquetPonceVanSchaftingen2013}.

\begin{proof}[Proof of Theorem~\ref{theorem:estimate_sgeq1}]
	We start by observing that it suffices to consider the case where the domain is \( \bR^{m} \).
	Indeed, if the domain is (the closure of) a smooth bounded open subset of \( \bR^{m} \), one may instead work with an extension of \( u \) to the whole \( \bR^{m} \).
	If the domain is more generally a compact manifold of dimension \( m \), one may localize the argument in a finite number of chart domains to return to the case of a smooth bounded open subset of \( \bR^{m} \).
	When \( s \in \bN_{\ast} \), the validity of this trick directly follows from the additivity of the integral.
	When \( s \) is not an integer, one should rely on a suitable substitute, see e.g.~\cite[Lemma~2.1]{Detaille2023}.
	Finally, by a standard approximation argument, we may assume that the map \( u \) under consideration is smooth.
	Indeed, as we are working with \( \bR^{\nu} \)\=/valued maps, we can then deduce the conclusion in the general case by approximating \( u \) with a sequence of smooth maps, extracting a subsequence converging almost everywhere, and using the lower-semicontinuity of the Sobolev seminorms combined with Fatou's lemma on the left-hand side of the estimate in the conclusion of the theorem.
	
	That being said, let us fix \( u \in \parens{W^{s,p} \cap L^{\infty} \cap C^{\infty}}\FromTo{\bR^{m}}{\bR^{\nu}} \), and write \( s = k + \sigma \), with \( k \in \bN \) and \( \sigma \in \coInterval{0}{1} \).
	The projected map \( P \circ \parens{u-a} \) is well-defined and smooth outside of \( u^{-1}\farg{\Sigma+a} \), which is a finite union of \( \parens{m-\ell} \)\=/submanifolds for almost every \( a \in \bR^{\nu} \), and hence in particular a null set as \( \ell \geq 1 \).
	In the rest of the proof, we shall implicitly restrict to the full measure sets of such values of \( a \) and \( x \) when writing pointwise estimates, for instance applying the Faà di Bruno formula.
	Moreover, when \( s \geq 1 \), we have \( \ell \geq 2 \), so that Lemma~\ref{lemma:small_sing_set_Sobolev} implies that it suffices to provide the Sobolev estimates for \( P \circ \parens{u-a} \), and this will automatically imply the Sobolev regularity and the fact that the weak derivatives coincide with the almost everywhere derivatives.
	
	We start with the case \( \sigma = 0 \), and thus \( s = k \in \bN_{\ast} \).
	The Faà di Bruno formula ensures that
	\begin{equation}
	\label{eq:ef572156ef556296}
		\abs{\Dext^{k} (P \circ (u-a))\farg{x}} 
		\lesssim
		\sum_{i=1}^{k}\sum_{\substack{1 \leq t_{1} \leq \cdots \leq t_{i} \\ t_{1} + \cdots + t_{i} = k}} \frac{1}{\dist{\farg{u\farg{x}-a,\Sigma}^{i}}}\abs{\Dext^{t_{1}}u(x)} \dotsb \abs{\Dext^{t_{i}}u\farg{x}}\text{.}
	\end{equation}
	Therefore, since \( ip \leq sp < \ell \), we deduce that
	\begin{equation}
		\label{eq:7c56cfc599ba65b1}
		\int_{B^{\nu}_{\alpha}} \abs{\Dext^{k} (P \circ (u-a))\farg{x}}^{p}\intd{a}
		\lesssim
		\sum_{i=1}^{k}\sum_{\substack{1 \leq t_{1} \leq \cdots \leq t_{i} \\ t_{1} + \cdots + t_{i} = k}} \abs{\Dext^{t_{1}}u(x)}^{p} \dotsb \abs{\Dext^{t_{i}}u\farg{x}}^{p}\text{.}
	\end{equation}
	Using Tonelli's theorem along with Hölder's inequality, we find
	\[
	\int_{B^{\nu}_{\alpha}}\parens[\bigg]{\int_{\bR^{m}}\abs{\Dext^{k} (P \circ (u-a))\farg{x}}^{p}}\intd{a}
	\lesssim
	\sum_{i=1}^{k}\sum_{\substack{1 \leq t_{1} \leq \cdots \leq t_{i} \\ t_{1} + \cdots + t_{i} = k}}\prod_{j = 1}^{i}\norm{\Dext^{t_{j}}u}_{L^{kp/t_{j}}\FromTo{\bR^{m}}{}}^{p}\text{.}
	\]
	Owing to the Gagliardo--Nirenberg inequality, we have \( u \in W^{t,\frac{kp}{t}}(\bR^{m}) \) for every \( 1 \leq t \leq k \) with
	\begin{equation*}
		\norm{\Dext^{t}u}_{L^{kp/t}\FromTo{\bR^{m}}{}}
		\lesssim
		\norm{u}_{L^{\infty}\FromTo{\bR^{m}}{}}^{1-\frac{t}{k}} \norm{u}_{W^{s,p}\FromTo{\bR^{m}}{}}^{\frac{t}{k}}\text{.}
	\end{equation*}
	We conclude that
	\[
	\int_{\ball{\alpha}{}{\nu}} \abs{P \circ (u-a)}_{W^{s,p}\FromTo{\bR^{m}}{}}^{p}\intd{a}
	\leq
	C\norm{u}_{W^{s,p}\FromTo{\bR^{m}}{}}^{p}\text{.}
	\]
	
	We now turn to the case \( 0 < \sigma < 1 \).
	By symmetry of the integrand in the Gagliardo seminorm with respect to \( x \) and \( y \), we only need to work on the region \( \set{\dist{\farg{u\farg{x}-a,\Sigma}} \leq \dist{\farg{u\farg{y}-a,\Sigma}}}{} \).
	We start by writing 
	\begin{multline}
	\label{eq:9ed9db683132939d}
		\abs{\Dext^{k}\parens{P \circ \parens{u-a}}\farg{x}-\Dext^{k}\parens{P \circ \parens{u-a}}\farg{y}}^{p} \\
		\lesssim
		\sum_{j=1}^{k}\sum_{\substack{1 \leq t_{1} \leq \cdots \leq t_{j} \\ t_{1} + \cdots + t_{j} = k}} \parens[\Big]{A_{j,t_{1},\dotsc,t_{j}}\farg{x,y,a}
		+
		B_{j,t_{1},\dotsc,t_{j}}\farg{x,y,a}}\text{,}
	\end{multline}
	where
	\begin{gather}
	\label{eq:2caebb62b1b4f646}
		\notag A_{j,t_{1},\dotsc,t_{j}}\farg{x,y,a}
		=
		\abs{\Dext^{j}P\farg{u\farg{x}-a}-\Dext^{j}P\farg{u\farg{y}-a}}^{p}\abs{\Dext^{t_{1}}u\farg{x}}^{p}\dotsb\abs{\Dext^{t_{1}}u\farg{x}}^{p} \\
		\intertext{and} 
		\begin{multlined}
			B_{j,t_{1},\dotsc,t_{j}}\farg{x,y,a} \\
			=
			\abs{\Dext^{j}P\farg{u\farg{y}-a}}^{p}\abs{\Dext^{t_{1}}u\farg{x} \otimes \dotsb \otimes \Dext^{t_{j}}u\farg{x} - \Dext^{t_{1}}u\farg{y} \otimes \dotsb \otimes \Dext^{t_{j}}u\farg{y}}^{p}\text{.}
		\end{multlined}
	\end{gather}
	
	We first handle the term containing \( A_{j,t_{1},\dotsc,t_{j}} \).
	Using~\cite[Lemma~3.13]{Detaille2025}, we estimate
	\begin{equation}
	\label{eq:ff7eec6a6ec04ece}
		\abs{\Dext^{j}P\farg{u\farg{x}-a}-\Dext^{j}P\farg{u\farg{y}-a}}
		\lesssim
		\frac{\abs{u\farg{x}-u\farg{y}}}{\dist{\farg{u\farg{x}-a,\Sigma}}^{j+1}}\text{.}
	\end{equation}
	On the other hand, by definition of singular projections,
	\begin{equation}
	\label{eq:72870f9447754116}
		\abs{\Dext^{j}P\farg{u\farg{x}-a}-\Dext^{j}P\farg{u\farg{y}-a}}
		\lesssim
		\frac{1}{\dist{\farg{u\farg{x}-a,\Sigma}}^{j}}\text{.}
	\end{equation}
	To estimate \( \int_{\ball{\alpha}{}{\nu}}\int_{\bR^{m}}\int_{\set{\dist{\farg{u\farg{x}-a,\Sigma}} \leq \dist{\farg{u\farg{y}-a,\Sigma}}}{}} \frac{A_{j,t_{1},\dotsc,t_{j}}\farg{x,y,a}}{\abs{x-y}^{m+\sigma p}} \intd{y}\intd{x}\intd{a} \), we split the integral with respect to \( y \) into two parts, one over the region \( \ball{\rho}{x}{m} \), and one over the region \( \bR^{m} \setminus \ball{\rho}{x}{m} \), with some \( \rho > 0 \) to be chosen later on.
	
	In the region \( \bR^{m} \setminus \ball{\rho}{x}{m} \), we rely on the straightforward estimate
	\begin{equation}
	\label{eq:57254e9b79390d2f}
		\int_{\bR^{m} \setminus \ball{\rho}{x}{m}} \frac{1}{\abs{x-y}^{m+\sigma p}}\intd{y}
		\lesssim
		\rho^{-\sigma p}\text{.}
	\end{equation}
	
	To estimate the integral over the region \( \ball{\rho}{x}{m} \), we start with the inequality
	\begin{equation}
	\label{eq:57254e9b79390d2g}
		\abs{u\farg{x}-u\farg{y}}^{p}
		\leq
		\abs{x-y}^{p}\int_{0}^{1}\abs{\Dext{u}\farg{x+t\parens{y-x}}}^{p}\intd{t}
		\quad
		\text{for almost every \( x \), \( y \in \bR^{m} \).}
	\end{equation}
	(This is straightforward when \( u \) is smooth, which is the case here.
	For a proof of the validity of~\eqref{eq:57254e9b79390d2g} for mere Sobolev maps, we refer the reader e.g.\ to~\cite[Proposition~1.4]{VanSchaftingenOxfordNotes}, which will turn useful in the proof of Theorem~\ref{theorem:conv_sgeq1} below.)
	With the change of variable \( h = y-x \), we obtain
	\begin{multline*}
		\int_{\ball{\rho}{x}{m}} \frac{\abs{u\farg{x}-u\farg{y}}^{p}}{\abs{x-y}^{m+\sigma p}}\intd{y}
		\leq
		\int_{0}^{1} \int_{\ball{\rho}{}{m}} \frac{\abs{\Dext{u}\farg{x+th}}^{p}}{\abs{h}^{m+(\sigma-1)p}}\intd{h}\intd{t} \\
		= 
		\int_{0}^{1} t^{(\sigma-1)p} \int_{\ball{t\rho}{x}{m}} \frac{\abs{\Dext{u}\farg{y}}^{p}}{\abs{x-y}^{m+(\sigma-1)p}}\intd{y}\intd{t}\text{.}
	\end{multline*}
	Hedberg's lemma~\cite{Hedberg1972} then ensures that
	\[
		\int_{\ball{\rho}{x}{m}} \frac{\abs{u\farg{x}-u\farg{y}}^{p}}{\abs{x-y}^{m+\sigma p}}\intd{y}
		\lesssim
		\int_{0}^{1} t^{(\sigma-1)p} (t\rho)^{(1-\sigma)p}\maximal{\abs{\Dext{u}}^{p}}\farg{x}\intd{t}
		=
		\rho^{(1-\sigma)p}\maximal{\abs{\Dext{u}}^{p}}\farg{x}\text{.}
	\]
	Choosing \( \rho = \dist{\farg{u\farg{x}-a,\Sigma}}\maximal{\abs{\Dext{u}}^{p}}\farg{x}^{-\frac{1}{p}} \), we conclude that
	\begin{multline}
	\label{eq:2905ae407e750248}
		\int_{\set{\dist{\farg{u\farg{x}-a,\Sigma}} \leq \dist{\farg{u\farg{y}-a,\Sigma}}}{}} \frac{A_{j,t_{1},\dotsc,t_{j}}\farg{x,y,a}}{\abs{x-y}^{m+\sigma p}} \intd{y} \\
		\begin{aligned}
			&\lesssim
			\parens[\bigg]{\rho^{-\sigma p}
			+
			\rho^{(1-\sigma)p}\maximal{\abs{\Dext{u}}^{p}}\farg{x}\frac{1}{\dist{\farg{u\farg{x}-a,\Sigma}}^{p}}}\frac{\abs{\Dext^{t_{1}}u\farg{x}}^{p} \cdots \abs{\Dext^{t_{j}}u\farg{x}}^{p}}{\dist{\farg{u\farg{x}-a,\Sigma}}^{jp}} \\
			&\lesssim
			\frac{1}{\dist{\farg{u\farg{x}-a,\Sigma}}^{sp}}\parens{\maximal{\abs{\Dext{u}}^{p}}\farg{x}}^{\sigma} \abs{\Dext^{t_{1}}u\farg{x}}^{p} \cdots \abs{\Dext^{t_{j}}u\farg{x}}^{p}\text{.}
		\end{aligned}
	\end{multline} 
	Here we have used the fact that \( \parens{j+\sigma}p \leq sp \) and the boundedness of \( u \).
	(The last hidden constant depends on the \( L^{\infty} \) norm of \( u \).)
	
	Using Tonelli's theorem and the fact that \( sp < \ell \), we deduce that
	\begin{multline*}
		\int_{\ball{\alpha}{}{\nu}}\int_{\bR^{m}}\int_{\set{\dist{\farg{u\farg{x}-a,\Sigma}} \leq \dist{\farg{u\farg{y}-a,\Sigma}}}{}} \frac{A_{j,t_{1},\dotsc,t_{j}}\farg{x,y,a}}{\abs{x-y}^{m+\sigma p}} \intd{y}\intd{x}\intd{a} \\
		\lesssim
		\int_{\bR^{m}} \parens{\maximal{\abs{\Dext{u}}^{p}}\farg{x}}^{\sigma} \abs{\Dext^{t_{1}}u\farg{x}}^{p} \cdots \abs{\Dext^{t_{j}}u\farg{x}}^{p}\intd{x}\text{.}
	\end{multline*}
	By Hölder's inequality, we obtain
	\[
		\int_{\bR^{m}} \parens{\maximal{\abs{\Dext{u}}^{p}}\farg{x}}^{\sigma} \abs{\Dext^{t_{1}}u\farg{x}}^{p} \cdots \abs{\Dext^{t_{j}}u\farg{x}}^{p}\intd{x}
		\leq
		\norm{\maximal{\abs{\Dext{u}}^{p}}}_{L^{s}\FromTo{\bR^{m}}{}}^{\sigma}\prod_{i=1}^{j}\norm{\Dext^{t_{i}}{u}}_{L^{sp/t_{i}}\FromTo{\bR^{m}}{}}^{p}\text{.}
	\]
	By virtue of the Gagliardo--Nirenberg inequality, we have \( \Dext^{t}u \in L^{sp/t}\FromTo{\bR^{m}}{} \) for every \( 1 \leq t \leq k \), with
	\begin{equation}
	\label{eq:ed0ef232cfebd49a}
		\abs{\Dext^{t}u}_{L^{sp/t}\FromTo{\bR^{m}}{}} \lesssim \norm{u}_{W^{s,p}(\FromTo{\bR^{m}}{}}^{\frac{t}{s}}\text{.}
	\end{equation} 
	In particular, since \( s > 1 \), the maximal function theorem implies that \( \maximal{\abs{\Dext{u}}^{p}} \in L^{s}\FromTo{\bR^{m}}{} \) with \( \norm{\maximal{\abs{\Dext{u}}^{p}}}_{L^{s}\FromTo{\bR^{m}}{}} \lesssim \norm{u}_{W^{s,p}\FromTo{\bR^{m}}{}}^{\frac{p}{s}} \).
	We conclude that 
	\[
		\int_{\bR^{m}} \parens{\maximal{\abs{\Dext{u}}^{p}}\farg{x}}^{\sigma} \abs{\Dext^{t_{1}}u\farg{x}}^{p} \cdots \abs{\Dext^{t_{j}}u\farg{x}}^{p}\intd{x}
		\lesssim
		\norm{u}_{W^{s,p}\FromTo{\bR^{m}}{}}^{p}\text{.}
	\]
	
	We now turn to the estimate of the second term in~\eqref{eq:9ed9db683132939d}, involving \( B_{j,t_{1},\dotsc,t_{j}}\farg{x,y,a} \).
	Using once again the assumption \( sp < \ell \) and Tonelli's theorem, we find 
	\begin{multline*}
		\int_{B^{\nu}_{\alpha}}\int_{\bR^{m}}\int_{\set{\dist{\farg{u\farg{x}-a,\Sigma}} \leq \dist{\farg{u\farg{y}-a,\Sigma}}}{}}\frac{B_{j,t_{1},\dotsc,t_{j}}\farg{x,y,a}}{\abs{x-y}^{m+\sigma p}}\intd{y}\intd{x}\intd{a} \\
		\lesssim
		\int_{\bR^{m}}\int_{\bR^{m}}\frac{\abs{\Dext^{t_{1}}u\farg{x} \otimes \dotsb \otimes \Dext^{t_{j}}u\farg{x} - \Dext^{t_{1}}u\farg{y} \otimes \dotsb \otimes \Dext^{t_{j}}u\farg{y}}^{p}}{\lvert x-y \rvert^{m+\sigma p}}\intd{y}\intd{x}\text{.}
	\end{multline*}
	The right-hand side above is nothing else but
	\(
		\abs{\Dext^{t_{1}}u \otimes \dotsb \otimes \Dext^{t_{j}}u}_{W^{\sigma,p}\FromTo{\Omega}{}}^{p}\text{.}
	\)
	We estimate it via Lemma~\ref{lemma:Gagliardo_seminorm_product}.
	For this purpose, we choose the \( v_{\alpha} \) to be the \( \Dext^{t_{i}}u \), the \( q_{\alpha} \) to be \( \frac{sp}{t_{\alpha}} \), the \( r_{\alpha} \) to be \( \frac{sp}{t_{\alpha+1}} \), and we deduce that
	\begin{multline*}
		\int_{\bR^{m}}\int_{\bR^{m}}\frac{\abs{\Dext^{t_{1}}u\farg{x} \otimes \dotsb \otimes \Dext^{t_{j}}u\farg{x} - \Dext^{t_{1}}u\farg{y} \otimes \dotsb \otimes \Dext^{t_{j}}u\farg{y}}^{p}}{\lvert x-y \rvert^{m+\sigma p}}\intd{y}\intd{x} \\
		\lesssim
		\sum_{i=1}^{j}\parens[\bigg]{\norm{\Dext^{t_{i}}{u}}_{L^{sp/t_{i}}\FromTo{\bR^{m}}{}}^{\parens{1-\sigma}p}\norm{\Dext^{t_{i}+1}{u}}_{L^{sp/\parens{t_{i}+1}}\FromTo{\bR^{m}}{}}^{\sigma p}\prod_{\substack{\beta=1 \\ \beta \neq i}}^{j}\norm{\Dext^{t_{\beta}}u}_{L^{sp/t_{\beta}}\FromTo{\bR^{m}}{}}^{p}}\text{.}
	\end{multline*}
	Invoking once more the Gagliardo--Nirenberg inequality~\eqref{eq:ed0ef232cfebd49a}, we deduce that
	\[
		\int_{\bR^{m}}\int_{\bR^{m}}\frac{\abs{\Dext^{t_{1}}u\farg{x} \otimes \dotsb \otimes \Dext^{t_{j}}u\farg{x} - \Dext^{t_{1}}u\farg{y} \otimes \dotsb \otimes \Dext^{t_{j}}u\farg{y}}^{p}}{\lvert x-y \rvert^{m+\sigma p}}\intd{y}\intd{x}
		\lesssim
		\norm{u}_{W^{s,p}\FromTo{\bR^{m}}{}}^{p}\text{,}
	\]
	which concludes the proof.
\end{proof}

We conclude this section by presenting two applications of Theorem~\ref{theorem:estimate_sgeq1}, as announced in the introduction.
First, we state the following result concerning the extension of traces of manifold-valued maps in the range \( s+\frac{1}{p} \geq 1 \).

\begin{theorem}
	\label{theorem:extension}
	Assume that \( \mN \) admits an \( \ell \)-singular projection \( P \colon \bR^{\nu} \setminus \Sigma \to \mN \).
	If \( s \notin \bN \) is such that \( s+\frac{1}{p} \geq 1 \) and if \( sp+1 < \ell \), then every map \( u \in W^{s,p}(\partial\mM;\mN) \) admits an extension \( U \in W^{s+\frac{1}{p},p}(\mM;\mN) \) such that \( \tr U = u \) on \( \partial\mM \) and
	\[
		\norm{U}_{W^{s+\frac{1}{p},p}\FromTo{\mM}{}}
		\leq
		\norm{u}_{W^{s,p}\FromTo{\partial\mM}{}}\text{,}
	\]
	for some constant \( C > 0 \) depending on \( s \), \( p \), \( \mM \), and \( \mN \).
	Moreover, the extension may be chosen to be smooth inside \( \Int{\mM} \) except on an \( \parens{m-\ell} \)\=/dimensional closedly embedded submanifold of \( \Int{\mM} \).
\end{theorem}

In the range where \( s+\frac{1}{p} < 1 \), the conclusion of Theorem~\ref{theorem:extension} holds under the assumption that \( p < \ell \); see~\cite[Theorem~2.11]{Vincent2025} for a closely related result.
We note that this assumption is indeed more stringent, as then
\[
	sp+1
	=
	p\parens[\Big]{s+\frac{1}{p}}
	<
	p\text{.}
\]
It is an open question whether or not the conclusion of Theorem~\ref{theorem:extension} holds under the weaker assumption \( sp+1 \) in the whole range where \( 0 < s < +\infty \) is not an integer.
If \( \mN \) admits an \( \ell \)\=/singular projection, there is hope that this approach might still work, taking into account the fact that \( \bR^{\nu} \)\=/valued extensions can be constructed by convolution of an \( \mN \)\=/valued \( W^{s,p} \) map, even though it does not seem obvious how to exploit this information --- unlike for the approximation problem, see Corollary~\ref{corollary:density_class_R} and the comment that follows it.
The general case where the target admits no singular projection is completely open, even though there is a natural candidate for the necessary and sufficient condition for the extension of traces to hold, in the spirit of the case \( s = 1 - \frac{1}{p} \), as explained in the introduction.

\begin{proof}[Proof of Theorem~\ref{theorem:extension}]
	By the classical extension theory, we know that \( u \) admits a bounded extension \( v \in W^{s+\frac{1}{p},p}(\mM;\bR^{\nu}) \) such that \( \tr v = u \) on \( \partial\mM \), which is moreover smooth inside \( \Int{\mM} \) and satisfies
	\[
		\norm{v}_{W^{s+\frac{1}{p},p}\FromTo{\mM}{}}
		\lesssim
		\norm{u}_{W^{s,p}\FromTo{\partial\mM}{}}\text{.}
	\]
	Using Theorem~\ref{theorem:estimate_sgeq1}, we choose \( a \in B^{\nu}_{\alpha} \) such that \( P \circ(v-a) \in W^{s+\frac{1}{p},p}(\mM;\mN) \) and
	\[
		\norm{P \circ \parens{v-a}}_{W^{s+\frac{1}{p},p}\FromTo{\mM}{}}
		\lesssim
		\norm{v}_{W^{s+\frac{1}{p},p}\FromTo{\mM}{}}\text{.}
	\]
	If \( \alpha > 0 \) is chosen to be sufficiently small, then Proposition~\ref{prop:Inv_trans_proj} ensures that \( P\parens{\cdot-a}\restrfun{\mN} \) is a smooth diffeomorphism.
	The map \( U = \parens{P\parens{\cdot-a}\restrfun{\mN}}^{-1} \circ P \circ (u-a) \) satisfies the desired conclusion.
\end{proof}

In the higher order setting, this approach also allows to prescribe the trace of the derivatives.
To avoid excessive technicality, let us only sketch the case \( s = 2-\frac{1}{p} \).
We claim that the extension \( U \in W^{2,p}(\mM;\mN) \) of a map \( u \in W^{2-\frac{1}{p},p}(\partial\mM;\mN) \) can be chosen to further satisfy \( \tr \partial_{\textnormal{n}} U = u_{\textnormal{n}} \) for any given map \( u_{\textnormal{n}} \in W^{1-\frac{1}{p},p}\FromTo{\partial\mM}{\bR^{\nu}} \) such that \( u_{\textnormal{n}}\farg{x} \in T_{u\farg{x}}\mN \) for almost every \( x \in \partial\mM \), where \( T_{u\farg{x}}\mN \) denotes the tangent plane of \( \mN \) at the point \( u\farg{x} \) and \( \partial_{\textnormal{n}} \) the derivative in the direction normal to \( \partial\mM \).
In a more abstract fashion, the condition \( u_{\textnormal{n}}\farg{x} \in T_{u\farg{x}}\mN \) could be formulated by relying on the tangent bundle of \( \mN \).
However, unlike \( \mN \), \( T\mN \) is not compact, and hence the Sobolev space of mappings into \( T\mN \) depends on the choice of the embedding of \( T\mN \) into a Euclidean space as soon as \( s > 1 \).
We therefore avoid this issue by relying on the fact that, as \( \mN \) is embedded into \( \bR^{\nu} \), each of its tangent planes is also naturally embedded into \( \bR^{\nu} \). 

Let us explain how to adapt the proof of Theorem~\ref{theorem:extension} to obtain the additional conclusion on the trace of the extension.
By the classical linear theory, the extension \( v \in W^{2,p}(\mM;\bR^{\nu}) \) can be chosen to further satisfy \( \tr \partial_{\textnormal{n}} v = u_{\textnormal{n}} \).
But now, as
\[
	\Dext\parens{P\parens{z-a}\restrfun{\mN}}^{-1}
	=
	\parens{\Dext P\parens{z-a}\restrfun{\mN}}^{-1}
	\quad
	\text{on \( T_{z}\mN \),}
\]
it follows directly from the chain rule that the map \( U = \parens{P\parens{\cdot-a}\restrfun{\mN}}^{-1} \circ P \circ (u-a) \) satisfies the required additional conclusion.

Theorem~\ref{theorem:extension} and the previous remark provide a first partial answer to several questions stated notably in~\cite{MironescuVanSchaftingen2021Toul} concerning generalizations of the inverse trace theory for mappings into manifolds for \( s \neq 1 - \frac{1}{p} \), possibly with prescribed normal traces, although our results are restricted to the special situation where a singular projection is available, as already explained in the introduction.

\bigskip

Second, we explain how the proof of Theorem~\ref{theorem:estimate_sgeq1} can be adapted to yield a result suited for converging sequences, from which one may derive a strong density result for almost smooth maps.
More precisely, we prove the following theorem.

\begin{theorem}
\label{theorem:conv_sgeq1}
	Let \( P \colon \bR^{\nu} \setminus \Sigma \to \mN \) be an \( \ell \)\=/singular projection.
	If \( s \geq 1 \) and \( sp < \ell \), then for every map \( u \in W^{s,p}\FromTo{\mM}{\bR^{\nu}} \cap L^{\infty}\FromTo{\mM}{\bR^{\nu}} \), for every sequence \( \seq{u_{n}}{n \in \bN} \) that converges to \( u \) in \( W^{s,p}\FromTo{\mM}{\bR^{\nu}} \) and is uniformly bounded in \( L^{\infty}\FromTo{\mM}{\bR^{\nu}} \), and for every \( \alpha > 0 \), we have
	\[
		\int_{\ball{\alpha}{}{\nu}} \abs{P \circ \parens{u_{n}-a} - P \circ \parens{u-a}}_{W^{s,p}\FromTo{\mM}{}}^{p}\intd{a}
		\to 
		0
		\quad
		\text{as \( n \to +\infty \).}
	\]
\end{theorem}
\begin{proof}
	The proof follows closely the strategy in the proof of Theorem~\ref{theorem:estimate_sgeq1}.
	To avoid repeating most of the proof, we only indicate the main steps of the argument, and make use of several intermediate steps of the proof of Theorem~\ref{theorem:estimate_sgeq1}.
	Moreover, it suffices to prove that the conclusion of the theorem holds up to extraction of a subsequence, since one may then conclude its validity along the whole sequence by applying the subsequence principle.
	Finally, as in Theorem~\ref{theorem:estimate_sgeq1}, it suffices to consider the case where the domain is \( \bR^{m} \).
	
	Unlike in the proof of Theorem~\ref{theorem:estimate_sgeq1}, we cannot assume \( u \) and the \( u_{n} \) to be smooth, as deducing the general case from this by approximation would imply exchanging two limits.
	However, \emph{once the validity of Theorem~\ref{theorem:estimate_sgeq1} is established}, we know in particular that \( P \circ \parens{u-a} \in W^{s,p} \) for almost every \( a \).
	Hence, for such \( a \), the derivatives of \( P \circ \parens{u-a} \) can be computed almost everywhere via the Faà di Bruno formula, so that we may follow the proof of Theorem~\ref{theorem:estimate_sgeq1} and obtain all the corresponding almost everywhere estimates \emph{without assuming a priori that \( u \) is smooth}.
	(The only exception is \eqref{eq:57254e9b79390d2g}, for which we already provided a justification if \( u \) is not smooth.)
	The only point of attention concerns the set of values such that \( u\farg{x}-a \in \Sigma \), that we can no longer prove to be a submanifold for almost every \( a \) via Sard's theorem if \( u \) is not smooth.
	However, we can still prove it to be a null set by a straightforward Tonelli-type argument.
	Indeed, the set of \( \parens{x,a} \) such that \( u\farg{x}-a \in \Sigma \) is measurable (with respect to the product measure) as the inverse image of a measurable set by a measurable function, and for a fixed \( x \), the corresponding set of values of \( a \) is a translate of \( \Sigma \), hence a null set.
	By Tonelli's theorem, for almost every \( a \), the set of \( x \) such that \( u\farg{x}-a \in \Sigma \) is therefore a null set.
	Summarizing the discussion, excluding a null set of values of \( a \) and then a null set of values of \( x \), we may justify all the pointwise estimates in the proof of Theorem~\ref{theorem:estimate_sgeq1}, even if \( u \) and the \( u_{n} \) are not assumed to be smooth.
	
	That being said, let us first assume that \( s = k \) is an integer.
	Combining the Gagliardo--Nirenberg inequality with the partial converse of the dominated convergence theorem, we find the existence of maps \( g_{t} \in L^{kp/t}\FromTo{\bR^{m}}{} \) such that, up to extraction of a subsequence, \( \abs{\Dext^{t}{u_{n}} } \leq g_{t} \) almost everywhere for every \( 1 \leq t \leq k \) and every \( n \in \bN \).
	This combined with~\eqref{eq:ef572156ef556296} and the assumption \( kp < \ell \) implies the uniform integrability of the sequence
	\(
		\seq{\Dext^{k}\parens{P \circ \parens{u_{n}-a}}\farg{x}}{n \in \bN}
	\)
	with respect to the variable \( a \).
	As this sequence converges almost everywhere, up to a further extraction, to \( \Dext^{k}\parens{P \circ \parens{u-a}}\farg{x} \), we deduce from the Vitali convergence theorem that
	\[
		\int_{\ball{\alpha}{}{\nu}}\abs{\Dext^{k}\parens{P \circ \parens{u_{n}-a}}\farg{x} - \Dext^{k}\parens{P \circ \parens{u-a}}\farg{x}}^{p}\intd{a}
		\to
		0
		\quad
		\text{for almost every \( x \in \bR^{m} \).}
	\]
	But now,~\eqref{eq:7c56cfc599ba65b1} combined with the domination \( \abs{\Dext^{t}{u_{n}} } \leq g_{t} \) allows to conclude by Lebesgue's dominated convergence theorem and Tonelli's theorem.
	
	Let us now assume that \( \sigma \in \ooInterval{0}{1} \), and follow the same strategy.
	Using again the partial converse of the dominated convergence theorem as well as the Gagliardo--Nirenberg inequality, we deduce from~\eqref{eq:2905ae407e750248} the uniform equi-integrability of the
	\[
		\int_{\set{\dist{\farg{u_{n}\farg{x}-a,\Sigma}} \leq \dist{\farg{u_{n}\farg{y}-a,\Sigma}}}{}} \frac{A_{j,t_{1},\dotsc,t_{j},n}\farg{x,y,a}}{\abs{x-y}^{m+\sigma p}} \intd{y}
	\]
	with respect to \( a \).
	Here, we have denoted by \( A_{j,t_{1},\dotsc,t_{j},n} \) the quantity \( A_{j,t_{1},\dotsc,t_{j}} \) associated to \( u_{n} \).
	Another application of the partial converse of the dominated convergence theorem and the assumption \( sp < \ell \) provides the uniform equi-integrability of the
	\[
		\int_{\set{\dist{\farg{u_{n}\farg{x}-a,\Sigma}} \leq \dist{\farg{u_{n}\farg{y}-a,\Sigma}}}{}} \frac{B_{j,t_{1},\dotsc,t_{j},n}\farg{x,y,a}}{\abs{x-y}^{m+\sigma p}} \intd{y}
	\]
	with respect to \( a \).
	Therefore, we are again in position to apply the Vitali convergence theorem to obtain
	\[
		\int_{\ball{\alpha}{}{\nu}}\int_{\set{\dist{\farg{u_{n}\farg{x}-a,\Sigma}} \leq \dist{\farg{u_{n}\farg{y}-a,\Sigma}}}{}}
		\frac{\abs{\Dext^{k}v_{n}\farg{x,a}-\Dext^{k}v_{n}\farg{y,a}+\Dext^{k}v\farg{x,a}-\Dext^{k}v\farg{y,a}}^{p}}{\abs{x-y}^{m+\sigma p}}\intd{y}\intd{a}
		\to
		0
	\]
	for almost every \( x \in \bR^{m} \), where \( v_{n}\farg{x,a} = P \circ \parens{u_{n}-a}\farg{x} \) and \( v\farg{x,a} = P \circ \parens{u-a}\farg{x} \).
	We conclude once again via the dominated convergence theorem.
\end{proof}

As a corollary, we deduce the following strong approximation result by almost smooth maps.

\begin{corollary}
\label{corollary:density_class_R}
	Assume that \( s \geq 1 \) and that \( \mN \) admits an \( \parens{\floor{sp}+1} \)\=/singular projection.
	Then, the class \( \Rbase_{s,p}\FromTo{\mM}{\mN} \) is dense in \( W^{s,p}\FromTo{\mM}{\mN} \).
\end{corollary}
\begin{proof}
	Given a map \( u \in W^{s,p}\FromTo{\mM}{\mN} \), we apply Theorem~\ref{theorem:conv_sgeq1} to any sequence \( \seq{u_{n}}{n \in \bN} \) of smooth bounded \( W^{s,p}\FromTo{\mM}{\bR^{\nu}} \) maps that converges to \( u \) in \( W^{s,p} \) and is uniformly bounded in \( L^{\infty} \), and we conclude with the aid of with Lemma~\ref{lemma:Sard} and Proposition~\ref{prop:Inv_trans_proj}.
\end{proof}

Corollary~\ref{corollary:density_class_R} is a special case of~\cite[Theorem~2.4]{Detaille2025}.
The main difference here is the restriction \( s \geq 1 \), while~\cite[Theorem~2.4]{Detaille2025} covers the full range \( 0 < s < +\infty \).
This comes from the fact that the proof of Corollary~\ref{corollary:density_class_R} relies on a result for the method of singular projection that works for \emph{any} converging sequence of \( \bR^{\nu} \)\=/valued \( W^{s,p} \) maps, while the approach in~\cite{Detaille2025} was relying heavily on the use of a sequence of maps obtained as convolutions of an \( \mN \)\=/valued map.

\section{The case \texorpdfstring{\( 0 < s < 1 \)}{0 < s < 1}}
\label{sect:0lesle1}

We now move to our negative result in the range \( 0 < s < 1 \), namely Theorem~\ref{theorem:ctrex_sle1}.
We start by introducing some notation.
In this section, as we are concerned with Theorem~\ref{theorem:ctrex_sle1}, we denote by \( P \) the retraction \( \bR^{\ell} \setminus \set{0}{} \to \bS^{\ell-1} \) given by \( P\farg{x} = \frac{x}{\abs{x}} \).
We let \( e_{i} \) be the \( i \)\=/th vector of the canonical basis of \( \bR^{\ell} \), and given \( c \in \ball{}{}{\ell} \) and \( n \in \bN \), we let \( c^{+} = c^{+}_{n} = c + 2^{1-n}e_{1} \) and \( c^{-} = c^{-}_{n} = c - 2^{1-n}e_{1} \).
Finally, similar to balls, we denote by \( \cube{r}{a}{\ell} \) the cube of inradius \( r \) centered at \( a \) in \( \bR^{\ell} \), so that \( \cube{r}{a}{\ell} \) has a sidelength equal to \( 2r \). 

The building block of our construction is a function taking both values \( c^{+} \) and \( c^{-} \) on two close cubes, and whose instrumental properties are collected in the following lemma.

\begin{lemma}
\label{lemma:patch}
	For every \( n \in \bN \), every \( c \in \ball{}{}{\ell} \), and every \( 0 < s < 1 \) and \( 1 \leq p < +\infty \), there exists a map \( v_{c} \in C^{\infty}\cs\FromTo{\bR^{\ell}}{\bR^{\ell}} \) with support contained in \( \cube{}{}{\ell} \) and such that
	\begin{equation}
	\label{eq:969c30cd8dd53af8}
		\abs{v_{c}}_{W^{s,p}\FromTo{\bR^{\ell}}{}}^{p} \leq C
	\end{equation}
	while, for every \( a \in \ball{}{}{\ell} \),
	\begin{equation}
	\label{eq:969c30cd8dd53af9}
		\abs{P \circ \parens{v_{c}-a}}_{W^{s,p}\FromTo{\ball{}{}{\ell}}{}}^{p}
		\geq
		C'\frac{\abs{P\farg{c^{+}-a}-P\farg{c^{-}-a}}^{p}}{2^{\parens{1-n}p}}\text{,}
	\end{equation}
	for some constants \( C \), \( C' > 0 \) depending on \( s \), \( p \), and \( \ell \).
\end{lemma}
\begin{proof}
	We fix a cutoff function \( \psi \in C^{\infty}\cs\FromTo{\bR^{\ell}}{} \) whose support is contained in \( \ball{}{}{\ell} \) and such that \( 0 \leq \psi \leq 1 \) and \( \psi = 1 \) on \( \ball{1/2}{}{\ell} \).
	We define the map \( w_{c} \) by \( w_{c}\farg{x} = c + 2^{1-n}\parens{\psi\farg{x-e_{1}}-\psi\farg{x+e_{1}}}e_{1} \), so that \( w_{c} = c^{-} \) on \( \ball{1/2}{-e_{1}}{\ell} \) and \( w_{c} = c^{+} \) on \( \ball{1/2}{e_{1}}{\ell} \).
	It is straightforward to observe that \( \abs{w_{c}}_{W^{s,p}\FromTo{\bR^{\ell}}{}} \lesssim 2^{\parens{1-n}p} \) while \( \abs{P \circ \parens{w_{c}-a}}_{W^{s,p}\FromTo{\ball{}{}{\ell}}{}}^{p} \gtrsim \abs{P\farg{c^{+}-a}-P\farg{c^{-}-a}}^{p} \).
	
	We would like to transform \( w_{c} \) into a compactly supported map while preserving the order of the ratio between the projected map and the original one.
	However, we need to cope with the fact that the transition towards \( 0 \) will bring a fixed cost to the Sobolev energy.
	We bypass this issue by clustering several scaled copies of the map \( w_{c} \), using an analogue argument to~\cite[Proof of Theorem~3.1]{MonteilVanSchaftingen2019}.
	Namely, we construct a map \( \tilde{w}_{c} \) by arranging scaled copies of the map \( w_{c} \) on \( k^{\ell} \) balls of radius of order \( \frac{1}{k} \) arranged in a regular grid inside \( \ball{1/2}{}{\ell} \), with \( k \in \bN_{\ast} \) to be chosen later on.
	In particular, the map \( \tilde{w}_{c} \) takes the value \( c \) outside of \( \ball{1/2}{}{\ell} \).
	By scaling and the countable patching property of Sobolev maps, it holds that
	\[
		\abs{\tilde{w}_{c}}_{W^{s,p}\FromTo{\bR^{\ell}}{}}^{p}
		\lesssim
		k^{\ell}k^{sp-\ell}2^{\parens{1-n}p}
		=
		k^{sp}2^{\parens{1-n}p}\text{,}
	\]
	while, by superadditivity,
	\[
		\abs{P \circ \parens{\tilde{w}_{c}-a}}_{W^{s,p}\FromTo{\ball{1/2}{}{\ell}}{}}^{p}
		\gtrsim
		k^{sp}\abs{P\farg{c^{+}-a}-P\farg{c^{-}-a}}^{p}\text{.}
	\]
	We choose \( k \) sufficiently large so that \( k^{sp} \simeq 2^{\parens{n-1}p} \).
	
	Finally, we choose a function \( \varphi_{c} \in C^{\infty}\cs\FromTo{\bR^{\ell}}{} \) whose support is contained in \( \ball{2}{}{\ell} \) and such that \( \varphi_{c} = c \) on \( \ball{}{}{\ell} \) and \( \abs{\varphi_{c}}_{W^{s,p}\FromTo{\bR^{\ell}}{}}^{p} \lesssim 1 \), and we define \( v_{c} \) by \( v_{c} = \tilde{w}_{c} \) on \( \ball{}{}{\ell} \) and \( v_{c} = \varphi_{c} \) outside of \( \ball{}{}{\ell} \).
	It follows from the almost additivity of the Gagliardo seminorm, see e.g.~\cite[Lemma~2.2]{MonteilVanSchaftingen2019}, that
	\[
		\abs{v_{c}}_{W^{s,p}\FromTo{\bR^{\ell}}{}}^{p}
		\lesssim
		\abs{\tilde{w}_{c}}_{W^{s,p}\FromTo{\bR^{\ell}}{}}^{p} + \abs{\varphi_{c}}_{W^{s,p}\FromTo{\bR^{\ell}}{}}^{p}
		\lesssim 1\text{,}
	\]
	while
	\[
		\abs{P \circ \parens{v_{c}-a}}_{W^{s,p}\FromTo{\ball{}{}{\ell}}{}}^{p}
		\gtrsim
		k^{sp}\abs{P\farg{c^{+}-a}-P\farg{c^{-}-a}}^{p}
		\gtrsim
		\frac{\abs{P\farg{c^{+}-a}-P\farg{c^{-}-a}}^{p}}{2^{\parens{1-n}p}}\text{,}
	\]
	which concludes the proof.
\end{proof}

To make use of Lemma~\ref{lemma:patch}, we need to suitably estimate the quantity \( \abs{P\farg{c^{+}-a}-P\farg{c^{-}-a}} \).
We present two different arguments: the first one is simpler, but it only covers the range \( p > \ell \); the second one is slightly more involved, but allows to handle also the limiting case \( p = \ell \).
These arguments each rely on a lemma involving planar geometry, respectively Lemma~\ref{lemma:geom1} and Lemma~\ref{lemma:geom2} below.

\begin{lemma}
\label{lemma:geom1}
	For every \( n \in \bN \), every \( c \in \ball{}{}{\ell} \), and every \( a \in \ccube{2^{-n}}{c}{\ell} \), we have
	\[
		\abs{P\farg{c^{+}-a}-P\farg{c^{-}-a}}
		\geq
		C\text{,}
	\]
	for some constant \( C > 0 \) depending only on \( \ell \).
\end{lemma}

\begin{lemma}
\label{lemma:geom2}
	For every \( n \in \bN \), every \( c \in \ball{}{}{\ell} \), and every \( a \in \ball{}{}{\ell} \) such that
	\begin{enumerate*}[label=(\(\textup{\roman*}\))]
		\item\label{it:ass_varphi} the angle \( \varphi \) formed by the vectors \( a-c \) and \( e_{1} \) satisfies \( \abs{\cos\varphi} \leq \frac{1}{8} \), 
		\item\label{it:ass_dist} \( \abs{a-c} \geq 2^{-n} \), 
	\end{enumerate*}
	we have
	\[
		\abs{P\farg{c^{+}-a}-P\farg{c^{-}-a}}
		\geq
		C\frac{2^{1-n}}{\abs{a-c}}\text{,}
	\]
	for some constant \( C > 0 \) depending only on \( \ell \).
\end{lemma}

Taking Lemmas~\ref{lemma:geom1} and~\ref{lemma:geom2} for granted, we proceed with the proof of the main result of this section.

\begin{proof}[Proof of Theorem~\ref{theorem:ctrex_sle1}]
	As explained in the introduction, our goal is to apply the general nonlinear uniform boundedness principle~\cite[Theorem~1.6]{MonteilVanSchaftingen2019} with the energy \( \mG \) over \( \bR^{\ell} \) with state space \( \bR^{\ell} \) given by 
	\[
		\mG\farg{u,A}
		=
		\inf_{a \in \ball{}{}{\ell}} \cE^{s,p}\farg{P \circ \parens{u-a},A}\text{.}
	\]
	Here, \( \cE^{s,p} \) denotes the Sobolev energy defined as \( \cE^{s,p}\farg{u,A} = \abs{u}_{W^{s,p}\FromTo{A}{}}^{p} \).
	It is readily seen that such a quantity indeed satisfies the monotonicity, scaling, and superadditivity requirements in~\cite{MonteilVanSchaftingen2019}.
	We are going to construct a sequence \( \seq{u_{n}}{n \in \bN} \) in \( W^{s,p}\FromTo{\ball{}{}{\ell}}{\bR^{\ell}} \) such that
	\[
		\lim_{n \to +\infty}\frac{\mG\farg{u_{n},\ball{}{}{\ell}}}{\cE^{s,p}\farg{u_{n},\ball{}{}{\ell}}}
		=
		+\infty\text{,}
	\]
	and the conclusion will follow at once from the nonlinear uniform boundedness principle.
	
	The key idea of the proof is to suitably glue together different "patches", each of them bringing an important contribution for some values of \( a \).
	For every \( n \in \bN \), we consider the standard decomposition of the unit cube \( \cube{}{}{\ell} \) into \( 2^{n\ell} \) cubes of sidelength \( 2^{1-n} \).
	We let \( \seq{c_{n,k}}{1 \leq k \leq 2^{n\ell}} \) be an enumeration of the centers of those cubes.
	For every \( k \), we define 
	\[
		c_{n,k}^{+} = c_{n,k} + 2^{1-n}e_{1}
		\quad
		\text{and}
		\quad
		c_{n,k}^{-} = c_{n,k} - 2^{1-n}e_{1}\text{,}
	\]
	following the notation introduced before the statement of Lemma~\ref{lemma:patch}.
	We define
	\[
		v_{n,k} = v_{c_{n,k}}\text{,}
	\]
	where \( v_{c_{n,k}} \) is the map provided by Lemma~\ref{lemma:patch}.
	Combining~\eqref{eq:969c30cd8dd53af9} with Lemma~\ref{lemma:geom1}, we find that
	\begin{equation}
		\label{eq:6f1ff66dc8dd4b58}
		\cE^{s,p}\farg{P\circ \parens{v_{n,k}-a},\ball{}{}{\ell}}
		\gtrsim 2^{np}
		\quad
		\text{for every \( a \in \ccube{2^{-n}}{c_{n,k}}{\ell} \).}
	\end{equation}
	
	We now exploit this lower bound to obtain the required estimate for our counterexample.
	We construct a map \( v_{n} \) by gluing together translates of the \( v_{n,k} \), one for each \( 1 \leq k \leq 2^{n\ell} \), so that their supports do not overlap.
	A straightforward lower bound yields
	\[
		\cE^{s,p}\farg{P\circ \parens{v_{n}-a},\supp{v_{n}}}
		\gtrsim
		2^{np}
		\quad
		\text{for every \( a \in \ball{}{}{\ell}\),}
	\]
	by using the estimate~\eqref{eq:6f1ff66dc8dd4b58} for a \( k \) such that \( a \in \ccube{2^{-n}}{c_{n,k}}{\ell} \) --- this \( k \) is unique unless \( a \) is on the boundary of some cube of the grid.
	On the other hand, the countable patching property of Sobolev mappings, see e.g.~\cite[Lemma~2.3]{MonteilVanSchaftingen2019}, implies that
	\begin{equation}
	\label{eq:787965c4b52eafe1}
		\cE^{s,p}\farg{v_{n},\bR^{\ell}}
		\lesssim
		\sum_{1 \leq k \leq 2^{n\ell}}\cE^{s,p}\farg{v_{n,k},\bR^{\ell}}
		\lesssim
		2^{n\ell}
		\text{.}
	\end{equation}
	
	In the case where \( p > \ell \), we find
	\begin{equation}
	\label{eq:787965c4b52eafe0}
		\frac{\inf\limits_{a \in \ball{}{}{\ell}}\cE^{s,p}\farg{P\circ \parens{v_{n}-a},\supp{v_{n}}}}{\cE^{s,p}\farg{v_{n},\bR^{\ell}}}
		\to
		+\infty
		\quad
		\text{as \( n \to +\infty \).}
	\end{equation}
	We conclude by letting \( u_{n} \) be a scaled copy of \( v_{n} \) so that its support is contained in \( \ball{}{}{\ell} \) and invoking the nonlinear uniform boundedness principle.
	
	\bigskip
	
	If \( p = \ell \), one needs to rely on a more careful estimate, and to take into account the contribution of several centers, relying on Lemma~\ref{lemma:geom2} instead to deduce the required lower bound on the \( \mG \) energy of \( v_{n} \).
	More specifically, for a fixed point \( a \in \ball{}{}{\ell} \) and a fixed \( d \in 2^{1-n}\bN_{\ast} \), we invoke Lemma~\ref{lemma:geom2} with \( c = c_{n,k} \) for every \( k \) such that the \( \ell^{\infty} \) distance between the center of the cube containing \( a \) and \( c_{n,k} \) is \( d \) and such that \( \abs{\varphi-\frac{\pi}{2}} \) is sufficiently small, and we find
	\[
		\abs{P\farg{c_{n,k}^{+}-a}-P\farg{c_{n,k}^{-}-a}}
		\gtrsim
		\frac{2^{1-n}}{d}\text{,}
	\]
	using also the fact that \( d \simeq \abs{c_{n,k}-a} \).
	The number of corresponding such indices \( k \) behaves like \( \parens{2^{n-1}d}^{\ell-1} \) --- this can be seen as a discrete isoperimetric inequality.
	We take into account all these contributions, for \( d = j2^{1-n} \) with \( j \in \bN_{\ast} \) ranging from \( 1 \) to \( 2^{n-1} \), and we get from~\eqref{eq:969c30cd8dd53af9}
	\[
		\cE^{s,p}\farg{P\circ \parens{v_{n}-a},\supp{v_{n}}}
		\gtrsim
		2^{np}\sum_{j=1}^{2^{n-1}} \frac{1}{j^{p}}j^{\ell-1}
		\gtrsim
		2^{np}\ln{2^{n-1}}\text{.}
	\]
	Combining this with~\eqref{eq:787965c4b52eafe1} and the fact that \( p = \ell \), we deduce that we are again in position to apply the uniform nonlinear boundedness principle to the maps \( u_{n} \) defined as scaled copies of \( v_{n} \) so that their supports are contained in \( \ball{}{}{\ell} \), and this concludes the proof.
\end{proof}

Let us make a comment about the strategy for estimating the energy of \( P \circ \parens{v_{n}-a} \) in the case where \( p = \ell \) in the above proof.
The idea was to take into account the contribution to the energy of not only one "patch" but of many of them.
More specifically, we took into account the contribution of every patch whose associated center \( c_{n,k} \) lies in a cone with vertex \( a \) and whose aperture is small, but fixed independently of the scale \( n \).
Alternatively, we could have relied on a similar construction, but instead of using as a building block the maps \( v_{n,k} \) with values concentrated around the two points \( c_{n,k}^{+} \) and \( c_{n,k}^{-} \) which are aligned along the direction \( e_{1} \), using rather maps with values concentrated around \emph{several} points, pairwise aligned along different directions.
Using a number of directions possibly large, but fixed independently of the scale \( n \), it is possible to ensure that \emph{every} center then leads a significant contribution to the energy of the projected map, eliminating the need of selecting only points in a fixed cone.

From a geometric point of view, this discussion amounts to say that one may either consider all centers in a region small but of fixed size determined by one axis, or use a number of axis large but independent of the scale to be in position to consider all centers.

We now prove Lemma~\ref{lemma:geom1} and~\ref{lemma:geom2}.

\begin{proof}[Proof of Lemma~\ref{lemma:geom1}]
	We refer the reader to Figure~\ref{fig:geom1} for an illustration of our constructions and notation.
	\begin{figure}[t]
		\centering
		\includegraphics[page=1]{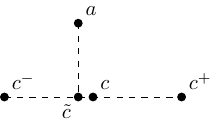}
		\caption{Estimate of \( \abs{P\farg{c^{+}-a}-P\farg{c^{-}-a}} \) in Lemma~\ref{lemma:geom1}}
		\label{fig:geom1}
	\end{figure}
	We let \( \tilde{c} \) be the orthogonal projection of \( a \) onto the line passing through \( c^{+} \) and \( c^{-} \), and we observe that \( \tilde{c} \in \ccube{2^{-n}}{c}{\ell} \).
	We compute
	\begin{equation}
		\label{eq:af4dde504e374905}
		P\farg{c^{+}-a}-P\farg{c^{-}-a}
		=
		\frac{\parens{c^{+}-a}\abs{c^{-}-a} - \parens{c^{-}-a}\abs{c^{+}-a}}{\abs{c^{+}-a}\abs{c^{-}-a}}\text{.}
	\end{equation}
	Let \( N \) denote the numerator of the right-hand side in~\eqref{eq:af4dde504e374905}.
	We compute
	\begin{equation}
		\label{eq:91b81d84429371dd}
		\abs{N}^{2}
		=
		2\abs{c^{+}-a}^{2}\abs{c^{-}-a}^{2} - 2\abs{c^{+}-a}\abs{c^{-}-a}\parens{c^{+}-a}\cdot\parens{c^{-}-a}\text{.}
	\end{equation}
	Writing the orthogonal decomposition \( c^{\pm}-a = c^{\pm}-\tilde{c}+\tilde{c}-a \), we find
	\[
	\parens{c^{+}-a}\cdot\parens{c^{-}-a}
	=
	\abs{\tilde{c}-a}^{2} -\abs{c^{+}-\tilde{c}}\abs{c^{-}-\tilde{c}}\text{,}
	\]
	where we have used the fact that \( c^{+}-\tilde{c} \) and \( c^{-}-\tilde{c} \) are aligned and have opposite directions.
	Since \( \tilde{c} \) is an orthogonal projection, we have \( \abs{\tilde{c}-a} \leq \abs{c^{\pm}-a} \).
	Hence, we find
	\[
	\abs{P\farg{c^{+}-a}-P\farg{c^{-}-a}}
	\geq
	2\frac{\abs{c^{+}-\tilde{c}}\abs{c^{-}-\tilde{c}}}{\abs{c^{+}-a}\abs{c^{-}-a}}\text{.}
	\]
	Finally, since \( a \in \ccube{2^{-n}}{c}{\ell} \), we find that \( \abs{c^{\pm}-\tilde{c}} \gtrsim 2^{-n} \) while \(\abs{c^{\pm}-a} \lesssim 2^{1-n} \), whence we conclude that
	\begin{equation*}
		\abs{P\farg{c^{+}-a}-P\farg{c^{-}-a}}
		\gtrsim
		1\text{,}
	\end{equation*}
	concluding the proof of the lemma.
\end{proof}

\begin{proof}[Proof of Lemma~\ref{lemma:geom2}]
	Since everything takes place inside one fixed plane --- depending on \( c \) and \( a \) --- we may assume that \( \ell = 2 \).
	We refer the reader to Figure~\ref{fig:geom2} for an illustration of our constructions and notation inside the plane.
	\begin{figure}[t]
		\centering
		\includegraphics[page=2]{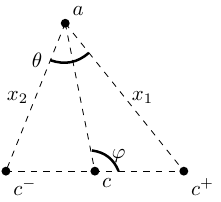}
		\caption{Estimate of \( \abs{P\farg{c^{+}-a}-P\farg{c^{-}-a}} \) in Lemma~\ref{lemma:geom2}}
		\label{fig:geom2}
	\end{figure}
	Let \( \theta \) denote the angle between the vectors \( c^{+}-a \) and \( c^{-}-a \), and let \( x_{1} = \abs{c^{+}-a} \) and \( x_{2} = \abs{c^{-}-a} \).
	By Al Kashi's formula, we know that
	\[
	\abs{P\farg{c^{+}-a}-P\farg{c^{-}-a}}^{2}
	=
	2 - 2\cos\theta\text{,}
	\]
	while 
	\[
	2^{2\parens{1-n}}
	=
	x_{1}^{2} + x_{2}^{2} - 2x_{1}x_{2}\cos\theta
	=
	\parens{x_{1}-x_{2}}^{2}
	+
	x_{1}x_{2}\parens{2 - 2\cos\theta}\text{.}
	\]
	Hence,
	\begin{equation}
		\label{eq:ce577bacc95dbcc2}
		\abs{P\farg{c^{+}-a}-P\farg{c^{-}-a}}^{2}
		=
		\frac{2^{2\parens{1-n}}-\parens{x_{1}-x_{2}}^{2}}{x_{1}x_{2}}\text{.}
	\end{equation}
	
	By another application of Al Kashi's formula, we find
	\begin{gather*}
		x_{1}^{2} = 2^{2\parens{1-n}} + \abs{a-c}^{2} - 2\cdot 2^{1-n}\abs{a-c}\cos{\varphi} \\
		\intertext{and}
		x_{2}^{2} = 2^{2\parens{1-n}} + \abs{a-c}^{2} + 2\cdot 2^{1-n}\abs{a-c}\cos{\varphi}\text{.}
	\end{gather*}
	Therefore, as either \( x_{1} \geq \abs{a-c} \) or \( x_{2} \geq \abs{a-c} \), we find
	\[
		\abs{x_{1}-x_{2}}\abs{a-c}
		\leq
		\abs{x_{1}^{2}-x_{2}^{2}}
		=
		4\cdot 2^{1-n}\abs{a-c}\abs{\cos{\varphi}}\text{,}
	\]
	so that, using~\ref{it:ass_varphi}, \( \abs{x_{1}-x_{2}} \leq \frac{1}{2}2^{1-n} \).
	Combining~\ref{it:ass_varphi} and~\ref{it:ass_dist}, we also find that \( x_{1} \simeq \abs{a-c} \simeq x_{2} \).
	We conclude from~\eqref{eq:ce577bacc95dbcc2} that
	\begin{equation*}
		\abs{P\farg{c^{+}-a}-P\farg{c^{-}-a}}
		\gtrsim
		\frac{2^{1-n}}{\abs{a-c}}\text{,}
	\end{equation*}
	hence finishing the proof of the lemma.
\end{proof}

In view of the counterexample provided by Theorem~\ref{theorem:ctrex_sle1}, the reader might wonder whether it would not be possible to work with a different singular projection, which would satisfy improved estimates near its singular set, to overcome the above issue.
We conclude this section with a short discussion that strongly suggests that there is no such hope.

We again work in the model case where \( \mN = \bS^{\ell-1} \subset \bR^{\ell} \).
In order to perform the approach by projection, we are looking for a retraction \( P \) into \( \bS^{\ell-1} \), which is defined on the whole \( \bR^{\ell} \) except on a finite set of point singularities.
When restricting to \( \cball{}{}{\ell} \), by a well-known  fact from topology, there must exist at least one singular point \( a \in \ball{}{}{\ell} \) of \( P \) such that, on any sufficiently small sphere centered at \( a \), \( P\) covers the whole \( \bS^{\ell-1} \).
It is then easy to show that 
\[
	\lim_{r \to 0} \sup_{x \in \partial\ball{a}{r}{}} \abs{x-a}\abs{\Dext{P}\farg{x}} \gtrsim 1\text{.}	
\]
In some sense, this means that the rate of blow up
\[
	\abs{\Dext{P}\farg{x}}
	\simeq
	\frac{1}{\abs{x-a}}
\]
near the singularity \( a \) that was required for \( P \) to be an \( \ell \)\=/singular projection cannot be improved.
As our proof of Theorem~\ref{theorem:ctrex_sle1} essentially relies on the estimate satisfied by \( P \), this suggests that there is no hope to find a singular projection of \( \bR^{\ell} \) onto \( \bS^{\ell-1} \) that would avoid the obstruction observed above, although we shall not attempt to present a complete argument here to avoid excessive technicality.

\section{The method of almost projection}
\label{sect:almost_projection}

In this last section, we explain how the idea used in the proof of Theorem~\ref{theorem:ctrex_sle1} can also be used to show that the \emph{method of almost projection} is also sensitive to the value of \( p \) instead of \( sp \) when \( 0 < s < 1 \). 
Let us first explain the key idea of this method, introduced by Haj\l asz~\cite{Hajlasz1994} when \( s = 1 \) and pursued by Bousquet, Ponce, and Van Schaftingen~\cite{BousquetPonceVanSchaftingen2013} when \( s \geq 1 \), and which has been applied so far to obtain strong and weak density results for smooth maps in Sobolev spaces of maps into manifolds that are \( \parens{\floor{sp}-1} \)\=/connected, respectively \( \parens{sp-1} \) connected with \( sp \in \bN_{\ast} \).
The starting point is again the non-existence of a globally defined retraction \( \upPi \colon \bR^{\nu} \to \mN \).
Instead of giving up the well-definedness of the map \( \upPi \) on the whole \( \bR^{\nu} \) as for the method of singular projection, the idea is to give up the fact that \( \upPi \) is a retraction on the whole \( \mN \).
That is, one is looking for a smooth map \( P \colon \bR^{\nu} \to \mN \) that coincides with \( \id_{\mN} \) on \( \mN \) \emph{except} on a small set where \( P\farg{x} \neq x \).

More precisely, it can be proved, see e.g.~\cite[Proposition~2.1]{BousquetPonceVanSchaftingen2013}, that if \( \mN \) is \( \parens{\ell-2} \)\=/connected, then for every \( \varepsilon > 0 \), there exists a smooth map \( P = P_{\varepsilon} \colon \bR^{\nu} \to \mN \) such that \( P = \id_{\mN} \) on \( \mN \setminus K \), where \( K = K_{\varepsilon} \) is a compact subset of \( \mN \) with \( \abs{K} \lesssim \varepsilon^{\ell-1} \), and such that 
\begin{equation}
\label{eq:c0bab852268dc089}
	\text{\( \norm{\Dext^{j}{P}}_{L^{\infty}} \lesssim \varepsilon^{-j} \) for every \( j \in \bN_{\ast} \).}
\end{equation}

To make our discussion more specific, we restrict to the case where \( \mN = \bS^{\ell-1} \subset \bR^{\ell} \).
In this situation, such a map \( P \) can be constructed by hand.
Indeed, one lets \( K = \cball{\varepsilon}{a}{} \) be the closure of the geodesic ball in \( \bS^{\ell-1} \) of radius \( \varepsilon \) and centered at \( a \), for any choice of \( a = a_{\varepsilon} \in \bS^{\ell-1} \), and defines \( P \) on \( K \) as a covering of \( \bS^{\ell-1} \setminus \Int{K} \) by \( K \) that coincides with the identity on \( \partial K \), obtained for instance via a stereographic projection; hence, \( P \) has degree \( 0 \), and can therefore be extended to an \( \bS^{\ell-1} \)\=/valued map defined on the whole \( \bR^{\ell} \).
For this construction, more can be said about the behavior of \( P \) on \( K \):
\begin{equation}
\label{eq:c0bab852268dc080}
	\abs{P\farg{x}-P\farg{y}} \simeq \varepsilon^{-1}\abs{x-y}
	\quad
	\text{for every \( x \), \( y \in \cball{\varepsilon/2}{a}{} \).}
\end{equation}

A crucial result concerning such a map \( P \) is the following, which relies once again on the Federer and Fleming averaging argument.
We denote by \( \overline{\upPi} \colon \bR^{\nu} \to \bR^{\nu} \) a chosen smooth extension --- non necessarily \( \mN \)\=/valued --- of the nearest point projection \( \upPi \colon \mN + \ball{\iota}{}{\nu} \to \mN \).
If \( s \geq 1 \) and \( sp \leq \ell-1 \), then for every \( \varepsilon \), there exists \( \xi_{\varepsilon} \) such that
\[
	\sup_{\varepsilon > 0} \abs{P_{\varepsilon} \circ \overline{\upPi} \circ  \parens{u-\xi_{\varepsilon}}}_{W^{s,p}\FromTo{\mM}{}}
	<
	+\infty
	\quad
	\text{for every \( u \in W^{s,p}\FromTo{\mM}{\mN} \).}
\]
If we actually have \( sp < \ell - 1 \), then it even holds that
\[
	\parens{\upPi\farg{\cdot-\xi_{\varepsilon}}\restrfun{\mN}}^{-1}\circ P_{\varepsilon} \circ \overline{\upPi} \circ  \parens{u-\xi_{\varepsilon}}
	\to
	u
	\quad
	\text{in \( W^{s,p}\FromTo{\mM}{\mN} \)  as \( \varepsilon \to 0 \).}
\]
With this at hand, it is easy to obtain a weak, respectively strong, approximation of \( u \) with smooth maps, by considering the map \( \parens{\upPi\farg{\cdot-\xi_{\varepsilon}}\restrfun{\mN}}^{-1}\circ P_{\varepsilon} \circ \overline{\upPi} \circ  \parens{u_{k}-\xi_{\varepsilon}} \), where \( \seq{u_{k}}{k \in \bN} \) is a sequence of \( \bR^{\nu} \)\=/valued smooth maps strongly converging to \( u \) in \( W^{s,p} \), and where \( k = k_{\varepsilon} \) is chosen sufficiently large via a diagonal argument, using the continuity of the composition operator on Sobolev spaces.

The following result shows that, in some sense, the method of almost projection is also sensitive to the more restrictive threshold \( p > \ell-1 \) when \( 0 < s < 1 \).

\begin{proposition}
\label{prop:ctrex_almost_projection}
	Assume that \( 0 < s < 1 \) and \( 1 \leq p < +\infty \) are such that \( sp < \ell-1 \) but \( p > \ell-1 \), and let \( P_{\varepsilon} \colon \bR^{\ell} \to \bS^{\ell-1} \) be a family of almost retractions satisfying the additional condition~\eqref{eq:c0bab852268dc080}.
	There exists a map \( u \in W^{s,p}\FromTo{\ball{}{}{\ell-1}}{\bS^{\ell-1}} \), independent of \( P_{\varepsilon} \), such that
	\[
		\sup_{\varepsilon > 0} \inf_{\xi \in \ball{\iota}{}{\ell}} \abs{P_{\varepsilon} \circ \overline{\upPi} \circ  \parens{u-\xi}}_{W^{s,p}\FromTo{\mM}{}}
		=
		+\infty\text{.}
	\]
\end{proposition}

Although the assumption~\eqref{eq:c0bab852268dc080} may seem to be more stringent than~\eqref{eq:c0bab852268dc089}, it is in practice satisfied by any almost retraction constructed in a \emph{reasonable} way, in particular using the general technique of construction provided, e.g., in the proof of~\cite[Proposition~2.1]{BousquetPonceVanSchaftingen2013} or in~\cite[\S2 to 4]{Hajlasz1994}.

We also observe that, unlike for Theorem~\ref{theorem:ctrex_sle1}, here our method does not allow to cover the limiting case \( p = \ell-1 \), in which we expect the method of almost projection to be applicable when \( 0 < s < 1 \), for instance in the context of weak density of smooth maps.

\begin{proof}
	As the construction is similar to the one in the proof of Theorem~\ref{theorem:ctrex_sle1}, we mostly focus on what should be changed in the argument.
	Also, since we are working here with \( \bS^{\ell-1} \)\=/valued maps, we abuse the language by saying that a map is \emph{compactly supported} whenever it is constant outside of a compact set.
	
	For some \( c > 0 \) to be chosen sufficiently small later on, we consider a collection \( \seq{c_{i}}{i \in I} \) of points in \( \bS^{\ell-1} \), with \( \abs{I} \simeq \varepsilon^{1-\ell} \), and such that any ball of radius \( c\varepsilon \) on \( \bS^{\ell-1} \) contains at least one \( c_{i} \).
	For each \( i \in I \), we pick two points \( c_{i}^{-} \), \( c_{i}^{+} \in \bS^{\ell-1} \) in \( \ball{c\varepsilon/2}{c_{i}}{} \) such that \( \dist\tuple{c_{i}^{-},c_{i}^{+}} \simeq \varepsilon \).
	We define the \( \bS^{\ell-1} \)\=/valued map \( v_{\varepsilon,i} \) by mimicking the construction of Lemma~\ref{lemma:patch}, replacing \( c^{\pm} \) by \( c_{i}^{\pm} \), and we let \( v_{\varepsilon} \) be defined by gluing together translated versions of the \( v_{\varepsilon,i} \), with \( i \in I \), so that their supports do not overlap.
	
	By the counterpart of~\eqref{eq:969c30cd8dd53af8} and the countable patching property of Sobolev mappings, it holds that
	\[
		\cE^{s,p}\farg{v_{\varepsilon},\bR^{\ell-1}}
		\lesssim
		\abs{I}
		\lesssim
		\varepsilon^{1-\ell}\text{.}
	\]
	On the other hand, if \( c \) is chosen sufficiently small, then for every \( a \in \bS^{\ell-1} \) and for every \( \xi \) with \( \abs{\xi} < \iota \), there exists at least one \( i \in I \) such that \( \overline{\upPi} \circ \parens{c_{i}^{\pm}-\xi} \in \ball{\varepsilon/2}{a}{} \) and the distance between both these points is of order \( \varepsilon \).
	Combining the counterpart of~\eqref{eq:969c30cd8dd53af9} with~\eqref{eq:c0bab852268dc080}, we find that
	\[
		\inf_{\xi \in \ball{\iota}{}{\ell}}\cE^{s,p}\farg{P_{\varepsilon} \circ \overline{\upPi} \circ  \parens{w_{\varepsilon}-\xi},\bR^{\ell-1}}
		\gtrsim
		\varepsilon^{-p}\text{.}
	\]
	
	To conclude the argument, we need to scale the maps \( v_{\varepsilon} \) to smaller balls, so that we can glue them together to form the required map \( u \).
	For this purpose, let \( \alpha > 1 \) to be chosen sufficiently close to \( 1 \) later on.
	We observe that the construction of the \( v_{\varepsilon} \) may be performed in such a way that its support is contained in a ball of radius of order \( \varepsilon^{-1} \).
	We define \( w_{\varepsilon}\farg{x} = v_{\varepsilon}\farg{x/\varepsilon^{\alpha\parens{\ell-1}/\parens{\ell-1-sp}}} \). 
	Hence,
	\begin{enumerate}[label=\((\textup{\roman*})\)]
		\item\label{it:rad_supp} the support of \( w_{\varepsilon} \) is contained in a ball of radius of order \( \varepsilon^{\alpha\parens{\ell-1}/\parens{\ell-1-sp}} \);
		\item\label{it:en_weps} \( \cE^{s,p}\farg{w_{\varepsilon},\bR^{\ell-1}} \lesssim \varepsilon^{\alpha\parens{\ell-1}}\varepsilon^{1-\ell} \);
		\item\label{it:en_proj_weps} \( \inf\limits_{\xi \in \ball{\iota}{}{\ell}}\cE^{s,p}\farg{P_{\varepsilon} \circ \overline{\upPi} \circ  \parens{w_{\varepsilon}-\xi},\bR^{\ell-1}}
		\gtrsim
		\varepsilon^{\alpha\parens{\ell-1}}\varepsilon^{-p} \).
	\end{enumerate}
	We choose \( \alpha > 1 \) sufficiently close to \( 1 \) so that \( \alpha\parens{\ell-1} - p < 0 \).
	We now define \( u \) by gluing together translated copies of the \( w_{\varepsilon_{n}} \) so that their supports do not overlap, where \( \varepsilon_{n} = 2^{-n} \).
	Property~\ref{it:rad_supp} ensures that this can be done in a ball of fixed radius --- up to an additional scaling, we may assume that \( u \) is supported in \( \ball{}{}{\ell-1} \).
	Property~\ref{it:en_weps} and the countable patching property of Sobolev maps ensure that \( u \in W^{s,p} \).
	Finally, property~\ref{it:en_proj_weps} ensures that 
	\[
		\sup_{\varepsilon > 0} \inf_{\xi \in \ball{\iota}{}{\ell}} \abs{P_{\varepsilon} \circ \overline{\upPi} \circ  \parens{u-\xi}}_{W^{s,p}\FromTo{\mM}{}}
		=
		+\infty\text{,}
	\]
	which concludes the proof of the proposition.
\end{proof}

We observe that, in the proof of Proposition~\ref{prop:ctrex_almost_projection}, the last step is very similar to the proof of the nonlinear uniform boundedness principle.
However, here this principle could not be applied as such, as the energy under consideration is not superadditive due to the presence of the \( \sup \) over \( \varepsilon \).
In our specific situation, it is nevertheless possible to mimic the proof of the nonlinear uniform boundedness principle without relying on the superadditivity assumption of the energy.
This is possible because the maps \( v_{\varepsilon} \) that we construct to contradict the linear growth of the energy with respect to the Sobolev energy, and which are to be scaled and glued together to reach the conclusion, can be chosen to be compactly supported, which removes the need of the \emph{clustering} step --- see~\cite{MonteilVanSchaftingen2019}, and also the proof of Lemma~\ref{lemma:patch} --- which is the place where this assumption is used.

\end{document}